\documentclass{article}
\usepackage{latexsym, a4wide}
\usepackage{amsmath, rotating}
\usepackage{amsfonts,amssymb, amsthm}
\usepackage{pictex,color}
\usepackage[normalem]{ulem}
\newcommand{\stackrell}[3]
{\begin{array}{c}\scriptstyle #1\\[-1ex] #2 \\[-1ex]\scriptstyle #3
\end{array} }
\newcommand\unnumberedfootnote[1]{ %
        \let\temp=\thefootnote %
        \renewcommand{\thefootnote}{}%
        \footnote{#1}%
        \let\thefootnote=\temp%
        \addtocounter{footnote}{-1}}

\newcommand{\sO}{\mathcal{O}}

\newtheorem{theorem}{Theorem}
\newtheorem{proposition}{Proposition}[section]
\newtheorem{lemma}[proposition]{Lemma}

\theoremstyle{definition}
\newtheorem{remark}[proposition]{Remark}

\numberwithin{equation}{section}
\usepackage{natbib}

\begin{document}

\title{\LARGE The pattern of genetic hitchhiking under recurrent
  mutation}

\author{\sc by Joachim Hermisson\thanks{University of Vienna,
    Nordbergstra\ss e 14, A-1090 Vienna, Austria\newline email:
    joachim.hermisson@math.univie.ac.at} and Peter
  Pfaffelhuber\thanks{Albert-Ludwigs University, Freiburg,
    Eckerstra\ss e 1, D-79104 Freiburg, Germany\newline email:
    peter.pfaffelhuber@stochastik.uni-freiburg.de}} \date{\today}
\maketitle
\unnumberedfootnote{\emph{AMS 2000 subject classification.} 92D15
  (Primary), 60J80, 60J85, 60K37, 92D10 (Secondary).}

\unnumberedfootnote{\emph{Keywords and phrases.} Selective sweep,
  genetic hitchhiking, soft selective sweep, diffusion approximation,
  Yule process, random background}

\begin{abstract}
\noindent
\emph{Genetic hitchhiking} describes evolution at a neutral locus that
is linked to a selected locus. If a beneficial allele rises to
fixation at the selected locus, a characteristic polymorphism pattern
(so-called selective sweep) emerges at the neutral locus. The
classical model assumes that fixation of the beneficial allele occurs
from a single copy of this allele that arises by mutation. However,
recent theory \citep{PenningsHermisson2006a, PenningsHermisson2006b}
has shown that recurrent beneficial mutation at biologically realistic
rates can lead to markedly different polymorphism patterns, so-called
\emph{soft selective sweeps}. We extend an approach that has recently
been developed for the classical hitchhiking model
\citep{SchweinsbergDurrett2005, EtheridgePfaffelhuberWakolbinger2006}
to study the recurrent mutation scenario. We show that the genealogy
at the neutral locus can be approximated (to leading orders in the
selection strength) by a marked Yule process with immigration. Using
this formalism, we derive an improved analytical approximation for the
expected heterozygosity at the neutral locus at the time of fixation
of the beneficial allele.
\end{abstract}


\section{Introduction}
The model of \emph{genetic hitchhiking}, introduced by
\cite{MaynardSmithHaigh1974}, describes the process of fixation of a
new mutation due to its selective advantage. During this fixation
process, linked neutral DNA variants that are initially associated
with the selected allele will \emph{hitchhike} and also increase in
frequency. As a consequence, sequence diversity in the neighborhood of
the selected locus is much reduced when the beneficial allele fixes, a
phenomenon known as a \emph{selective sweep}. This characteristic
pattern in DNA sequence data can be used to detect genes that have
been adaptive targets in the recent evolutionary history by
statistical tests (e.g. \citealt{KimStephan2002, NielsenEtAl2005,
  JensenThorntonBustamanteAquadro2007}).

Since its introduction, several analytic approximations to quantify
the hitchhiking effect have been developed
\citep{KaplanHudsonLangley1989, StephanWieheLenz1992, Barton1998,
  SchweinsbergDurrett2005, EtheridgePfaffelhuberWakolbinger2006,
  ErikssonFernstromMehligSagitiv2008}. The mathematical analysis of
selective sweeps makes use of the coalescent framework
\citep{Kingman1982, Hudson1983}, which describes the genealogy of a
population sample backward in time. Most studies follow the suggestion
of \citet[]{KaplanHudsonLangley1989} and use a structured coalescent
to describe the genetic footprint at a linked neutral locus,
conditioned on an approximated frequency path of the selected
allele. In this approach, population structure at the neutral locus
consists of the wild-type and beneficial background at the selected
locus, respectively. A mathematical rigorous construction was given by
\citet{BartonEtheridgeSturm2004}. Moreover, a structured ancestral
recombination graph was used in \citet{PfaffelhuberStudeny2007,
  McVean2007, PfaffelhuberLehnertStephan2008} to describe the common
ancestry of two neutral loci linked to the beneficial allele.

It has long been noted that the initial rise in frequency of a
beneficial allele is similar to the evolution of the total mass of a
supercritical branching process (\citealt{Fisher1930,
  KaplanHudsonLangley1989, Barton1998}; \citealt{Ewens2004},
p. 27f). This insight led to the approximation of the structured
coalescent by the genealogy of a supercritical branching process---a
Yule process \citep{OConnell1993, EvansOConnell1994}. Given a
selection intensity of $\alpha$ and a recombination rate of $\rho$
between the selected and neutral locus, it has been shown that a Yule
process with branching rate $\alpha$, which is marked at rate $\rho$
and stopped upon reaching $\lfloor 2\alpha\rfloor$ lines, is an
accurate approximation of the structured coalescent
\citep{SchweinsbergDurrett2005, EtheridgePfaffelhuberWakolbinger2006,
  PfaffelhuberHauboldWakolbinger2007}. For the standard scenario of
genetic hitchhiking, this approach leads to a refined analytical
approximation of the sampling distribution, estimates of the
approximation error and to efficient numerical simulations.

The classical hitchhiking model assumes that adaptation occurs from a
single origin of the beneficial allele. An explicit mutational process
at the selected locus, where the beneficial allele can enter the
population recurrently, is not taken into account. However, it has
recently been demonstrated that recurrent beneficial mutation at a
biologically realistic rate can lead to considerable changes in the
selective footprint in DNA sequence data \citep{HermissonPennings2005,
  PenningsHermisson2006a, PenningsHermisson2006b}. In the present
paper, we extend the Yule process approach of
\citet{EtheridgePfaffelhuberWakolbinger2006} to the full biological
model with recurrent mutation at the beneficial locus. Specifically,
we show that the genealogy at the selected site can be approximated by
a Yule process with immigration. Our results can serve as a basis for
a detailed analysis of patterns of genetic hitchhiking under recurrent
mutation, such as the site-frequency spectrum and linkage
disequilibrium patterns. As an example of such an application, we
derive the expected heterozygosity in Section \ref{Sub:het}.

~~

The paper is organized as follows. In Section \ref{S:model}, we
introduce the model as well as the structured coalescent and we
discuss the biological context of our work. In Section \ref{S:results}
we state results on the adaptive process, give the approximation of
the structured coalescent by a Yule process with immigration and apply
the approximation to derive expressions for the heterozygosity at the
neutral locus at the time of fixation. In Sections \ref{sub:P1},
\ref{sub:key} and \ref{sub:T1} we collect all proofs.

\section{The model}
\label{S:model}
We describe evolution in a two-locus system, where a neutral locus is
linked to a locus experiencing positive selection. In Section
\ref{sub:diff}, we first focus on the selected locus and formulate the
adaptive process as a diffusion. In Section \ref{sub:coal}, we
describe the genealogy at the neutral locus by a structured
coalescent. In Section \ref{sub:biol} we discuss the biological
context.

\subsection{Time-forward process}
\label{sub:diff}
Consider a population of constant size $N$. Individuals are haploid;
their genotype is thus characterized {\color{black}by a single copy of
  each allele}. Selection acts on a single bi-allelic locus. The
ancestral (wild-type) allele $b$ has fitness $1$ and the beneficial
variant $B$ has fitness $1+s$, where $s>0$ is the selection
coefficient.  Mutation from $b$ to $B$ is recurrent and occurs with
probability $u$ per individual per generation.  Let $X_t$ be the
frequency of the $B$ allele in generation $t$. In a standard
Wright-Fisher model with discrete generations, the number of
$B$-alleles in the offspring generation $t+1$ is $N X_{t+1}$, which is
binomially distributed with parameters $\frac{(1+s)X_t +
  u(1-X_t)}{(1+s)X_t+1-X_t}$ and $N$.


We assume that the beneficial allele $B$ is initially absent from the
population in generation $t = 0$ when the selection pressure on the
$B$ locus sets in. Since the $B$ allele is created recurrently by
mutation and we ignore back-mutations it will eventually fix at some
time $T$, i.e.\ $X_{t} = 1$ for $t\geq T$. This process of fixation
can be approximated by a diffusion.  To this end, let $\mathcal X^N =
(X_t^N)_{t=0,1,2,\ldots}$ with $X^N_0=0$ be the path of allele
frequencies of $B$.

Assuming $u=u_N\to 0, s=s_N\to 0$ such that $2Nu\to \theta,
Ns\to\alpha$ as $N\to\infty$, it is well-known (see e.g.\
\citealt{Ewens2004}) that $(X^N_{\lfloor 2Nt\rfloor})_{t\geq
  0}\Rightarrow (X_t)_{t\geq 0}$ as $N\to\infty$ where $\mathcal X :=
(X_t)_{t\geq 0}$ follows the SDE
\begin{equation}
\begin{aligned}\label{eq:diff}
  dX &= \big( \tfrac \theta 2 (1-X) + \alpha X (1-X) \big)dt +
  \sqrt{X(1-X)} dW
\end{aligned}
\end{equation}
with $X_0=0$. In other words, the diffusion approximation of $\mathcal
X^N$ is given by a diffusion $\mathcal X$ with drift and diffusion
coefficients
\begin{equation*}
  \begin{aligned}
    \mu_{\alpha,\theta}(x) = (\tfrac \theta 2 + \alpha x)(1-x), \qquad
    \sigma^2(x) =x(1-x).
  \end{aligned}
\end{equation*}
We denote by $\mathbb P^p_{\alpha,\theta}[.]$ and $\mathbb
E^p_{\alpha,\theta}[.]$ the probability distribution and its
expectation with respect to the diffusion with parameters
$\mu_{\alpha,\theta}$ and $\sigma^2$ 
{\color{black} and $X_0=p$ almost surely}. The fixation time can be
expressed in the diffusion setting as
\begin{align}\label{eq:defT}
  T:=\inf\{t\geq 0: X_t=1\}.
\end{align}

\subsection{Genealogies}
\label{sub:coal}
We are interested in the change of polymorphism patterns at a neutral
locus that is linked to a selected locus. We ignore recombination
within the selected and the neutral locus, but (with sexual
reproduction) there is the chance of recombination between the
selected and the neutral locus.  Let the recombination rate per
individual be $\rho$ in the diffusion scaling (i.e. $r=r_N$ is the
recombination probability in a Wright-Fisher model of size $N$ and
$r_N\xrightarrow{N\to\infty}0$ and $N
r_N\xrightarrow{N\to\infty}\rho$). Not all recombination events have
the same effect, however. We will be particularly interested in events
that change the genetic background of the neutral locus at the
selected site from $B$ to $b$, or vice-versa.  This is only possible
if $B$ individuals from the parent generation reproduce with $b$
individuals.  Under the assumption of random mating, the effective
recombination rate in generation $t$ that changes the genetic
background is thus $\rho X_t(1-X_t)$ in the diffusion setting.

Following \citet{BartonEtheridgeSturm2004}, we use the structured
coalescent to describe the polymorphism pattern at the neutral locus
in a sample. In this framework, the population is partitioned into two
demes according to the allele ($B$ or $b$) at the selected locus. The
relative size of these demes is defined by the fixation path $\mathcal
X$ of the $B$ allele. Only lineages in the same deme can
coalesce. Transition among demes is possible by either recombination
or mutation at the selected locus. We focus on the pattern at the time
$T$ of fixation of the beneficial allele. Throughout we fix a sample
size $n$.

\begin{remark}
  We define the coalescent as a process that takes values in
  partitions and introduce the following notation. Denote by
  $\Sigma_n$ the set of partitions of $\{1,...,n\}$. Each
  $\xi\in\Sigma_n$ is thus a set $\xi=\{\xi_1,...,\xi_{|\xi|}\}$ such
  that $\bigcup_{i=1}^{|\xi|} \xi_i = \{1,...,n\}$ and $\xi_i\cap
  \xi_j = \emptyset$ for $i\neq j$. Partitions can also be defined by
  equivalence relations and we write $k\sim_\xi\ell$ iff there is
  $1\leq i\leq |\xi|$ such that $k,\ell\in\xi_i$. Equivalently, $\xi$
  defines a map $\xi: \{1,...,n\}\to \{1,...,|\xi|\}$ by setting
  $\xi(k)=i$ iff $k\in\xi_i$.  We will also need the notion of a
  composition of two partitions.  If $\xi$ is a partition of
  $\{1,...,n\}$ and $\eta$ is a partition of $\{1,...,|\xi|\}$, define
  the partition $\xi\circ \eta$ on $\{1,...,n\}$ by
  $k\sim_{\xi\circ\eta}\ell$ iff $\xi(k)\sim_\eta \xi(\ell)$.\qed
\end{remark}

Setting $\beta=T-t$ we are interested in the genealogical process
$\xi^{\mathcal X}=(\xi_\beta)_{0\leq \beta\leq T}$ of a sample of size
$n$, conditioned on the path $\mathcal X$ of the beneficial allele
$B$.  The state space of $\xi^{\mathcal X}$ is
$$ 
S_n := 
\{(\xi^B, \xi^b): \xi^B\cup\xi^b \in \Sigma_n\}.
$$ 
Elements of $\xi^B$ ($\xi^b$) are ancestral lines of neutral loci that
are linked to a beneficial (wild-type) allele. Since there are only
beneficial alleles at time $T$, the starting configuration of
$\xi^{\mathcal X}$ is
$$ 
\xi^{\mathcal X}_0 = (\{1\},...,\{n\},\emptyset).
$$
For a given coalescent state $\xi^{\mathcal X}_\beta = (\xi^B, \xi^b)$
at time $\beta$, several events can occur, with rates that depend on
the value of the frequency path $\mathcal X$ at that time,
${X}_{T-\beta}$.  Coalescences of pairs of lines in the
beneficial (wild-type) background occur at rate $1/X_{T-\beta}$
($1/(1-X_{T-\beta})$).  Formally, for all pairs $1\leq i<j\leq
|\xi^B|$ and $1\leq i'<j'\leq |\xi^b|$, transitions occur at time
$\beta$ to
\begin{equation}\label{eq:rates1a}
\begin{aligned}
  \big((\xi^B\setminus \{\xi^B_i,\xi^B_j\})\cup
  \{\xi^B_i\cup\xi^B_j\},\xi^b\big) &
  \qquad \text{ with rate } \qquad  \frac{1}{X_{T-\beta}}\\
  \big(\xi^B, (\xi^b\setminus \{\xi^b_{i'},\xi^b_{j'}\})\cup
  \{\xi^b_{i'}\cup\xi^b_{j'}\}\big) & \qquad \text{ with rate } \qquad
  \frac{1}{1-X_{T-\beta}}.
\end{aligned}
\end{equation}
Changes of the genetic background happen either due to mutation at the
selected locus or recombination events between the selected and the
neutral locus. For $1\leq i\leq |\xi^B|$, transitions of genetic
backgrounds due to mutation occur at time $\beta$ from $\xi^{\mathcal
  X}_\beta =(\xi^B,\xi^b)$ for $1\leq i\leq |\xi^B|$ to
\begin{align}\label{eq:rates1}
  \big(\xi^B\setminus \{\xi^B_i\}, \xi^b\cup \{\xi^B_i\}\big)& \qquad
  \text{ with rate } \qquad \frac \theta
  2\frac{1-X_{T-\beta}}{X_{T-\beta}}.
\end{align}
(Recall that we assume that there are no back-mutations to the
wild-type). Moreover, changes of the genetic background due to
recombination occur at time $\beta$ for $1\leq i\leq |\xi^B|$, $1\leq
i'\leq |\xi^b|$ from $\xi^{\mathcal X}_\beta =(\xi^B,\xi^b)$ to
\begin{subequations}
  \label{eq:rates1b}
  \begin{align}\label{eq:rates1b1}
    \big(\xi^B\setminus \{\xi^B_i\}, \xi^b\cup \{\xi^B_i\}\big) &
    \qquad \text{ with rate } \qquad \rho
    (1-X_{T-\beta})\\ \label{eq:rates1b2} \big(\xi^B\cup
    \{\xi^b_{i'}\}, \xi^b\setminus \{\xi^b_{i'}\}\big) & \qquad \text{
      with rate } \qquad \rho X_{T-\beta}.
  \end{align}
\end{subequations}
All rates of $\xi^{\mathcal X}$ are collected in Table \ref{tab:2}.

\begin{table}
  \begin{center}
\begin{tabular}{|c|ccccc|}\hline
  \rule[-3mm]{0cm}{1cm}event & \hspace{0.5ex} coal in $B$ \hspace{0.5ex}& \hspace{0.5ex} coal in $b$ \hspace{0.5ex} &  \hspace{0.5ex} mut from $B$ to $b$\hspace{0.5ex} & \hspace{0.5ex} rec from $B$ to $b$ \hspace{0.5ex} & \hspace{0.5ex} rec from $b$ to $B$   \\[2ex]\hline
  \rule[-3mm]{0cm}{1cm}rate & $\frac{1}{X_t}$ & $\frac{1}{1-X_t}$ & $\frac \theta 2\frac{1-X_t}{X_t}$ & $\rho(1-X_t)$ &  $\rho X_t$ \\\hline
\end{tabular}
\end{center}
\caption{\label{tab:2}Transition rates in the process $\xi^{\mathcal X}$ at time $t=T-\beta$. Coalescence rates are equal for all pairs of partition elements in the beneficial and wild-type background. 
  Recombination and mutation rates are equal for all partition elements in $\xi^{B}$ and $\xi^{b}$.}
\end{table}

\begin{remark}
\mbox{}
  \begin{enumerate}
  \item The rates for mutation and recombination can be understood
    heuristically. Assume $X_t^N=x$ and assume $u,s,r$ are small. A
    neutral locus linked to a beneficial allele in generation $t+1$
    falls into one of three classes: (i) the {\color{black}class} for
    which the ancestor of the selected allele was beneficial has
    frequency $x + \mathcal O(u,s,r)$; (ii) the {\color{black}class}
    for which the beneficial allele was a wild-type and mutated in the
    last generation has frequency $u(1-x) + \mathcal O(us, ur)$; (iii)
    the {\color{black}class} for which the neutral locus was linked to
    a wild-type allele in generation $t$ and recombined with a
    beneficial allele has frequency $rx(1-x)+\mathcal
    O(ru,rs)$). Hence, if we are given a neutral locus in the
    beneficial background, the probability that its linked selected
    locus experienced a mutation one generation ago is
    $\frac{u(1-x)}{x} + \mathcal O(u^2, us,ur)$ and that it recombined
    with a wild-type allele one generation ago is $\frac{rx(1-x)}{x} +
    \mathcal O(ru, rs,r^2)$.  Thus, the rates \eqref{eq:rates1} and
    \eqref{eq:rates1b1} arise by a rescaling of time by $N$.
  \item In \eqref{eq:rates1a} and \eqref{eq:rates1} the rates have
    singularities when $X_{T-\beta}=0$. However, we will show in Lemma
    \ref{l:est2} using arguments from \cite{BartonEtheridgeSturm2004}
    and \cite{Taylor2007} that a line will almost surely leave the
    beneficial background \emph{before} such a singularity occurs. In
    particular, the structured coalescent process $\xi^{\mathcal X}$
    is well-defined.
  \end{enumerate}
\end{remark}


\subsection{Biological context}
\label{sub:biol}
A \emph{selective sweep} refers to the reduction of sequence diversity
and a characteristic polymorphism pattern around a positively selected
allele.  Models show that this pattern is most pronounced close to the
selected locus if selection is strong and if the sample is taken in a
short time window after the fixation of the beneficial allele (i.e.\
before it is diluted by new mutations). Today, biologists try to
detect sweep patterns in genome-wide polymorphism scans in order to
identify recent adaptation events (e.g. \citealp{HarrEtAl2002,
  OmettoEtAl2005, WilliamsonEtAl2005}).

The detection of sweep regions is complicated by the fact that certain
demographic events in the history of the population (in particular
bottlenecks) can lead to very similar patterns. Vice-versa, also the
footprint of selection can take various guises.  In particular, recent
theory shows that the pattern can change significantly if the
beneficial allele at the time of fixation traces back to more than a
single origin at the start of the selective phase (i.e.\ there is more
than a single ancestor at this time). As a consequence, genetic
variation that is linked to any of the successful origins of the
beneficial allele will survive the selective phase in proximity of the
selective target and the reduction in diversity (measured e.g. by the
number of segregating sites or the average heterozygosity in a sample)
is less severe.  \citet{PenningsHermisson2006a} therefore called the
resulting pattern a \emph{soft selective sweep} in distinction of the
classical \emph{hard sweep} from only a single origin. Nevertheless,
also a soft sweep has highly characteristic features, such as a more
pronounced pattern of \emph{linkage disequilibrium} as compared to a
hard sweep \citep{PenningsHermisson2006b}.

Soft sweeps can arise in several biological scenarios. For example,
multiple copies of the beneficial allele can already segregate in the
population at the start of the selective phase (adaptation from
standing genetic variation; \citealt{HermissonPennings2005,
  Przeworski2005}). Most naturally, however, the mutational process at
the selected locus itself may lead to a recurrent introduction of the
beneficial allele. Any model, like the one in this article, that
includes an explicit treatment of the mutational process will
therefore necessarily also allow for soft selective sweeps. For
biological applications the most important question then is:
\emph{When are soft sweeps from recurrent mutation likely?}  The
results of \cite{PenningsHermisson2006a} as well as Theorem \ref{T1}
in the present paper show that the probability of soft selective
sweeps is mainly dependent on the population-wide mutation rate
$\theta$. The classical results of a \emph{hard sweep} are reproduced
in the limit $\theta\to 0 $ and generally hold as a good approximation
for $\theta < 0.01$ in samples of moderate size. For larger $\theta$,
approaching unity, soft sweep phenomena become important.

Since $\theta$ scales like the product of the (effective) population
size and the mutation rate per allele, soft sweeps become likely if
either of these factors is large. Very large population sizes are
primarily found for insects and microbial organisms. Consequently,
soft sweep patterns have been found, e.g., in {\em Drosophila}
\citep{SchlenkeBegun2004} and in the malaria parasite {\em Plasmodium
  falsiparum} \citep{NairEtAl2007}.  Since point mutation rates
(mutation rates per DNA base per generation per individual) are
typically very small ($\sim 10^{-8}$), large mutation rates are
usually found {\color{black}in situations where many possible mutations
  produce the same (i.e.\ physiologically equivalent) allele}. This
holds, in particular, for adaptive loss-of-function mutations, where
many mutations can destroy the function of a gene. An example is the
loss of pigmentation in {\em Drosophila santomea}
\citep{JeongEtAl2008}.  But also adaptations in regulatory regions
often have large mutation rates and can occur recurrently. A
well-known example is the evolution of adult lactose tolerance in
humans, where several mutational origins have been identified
\citep{TishkoffEtAl2007}.

Several extensions of the model introduced in Section \ref{S:model}
are possible. In a full model, we should allow for the possibility of
back-mutations from the beneficial to the wild-type allele in natural
populations. However, such events are rarely seen in any sample
because such back-mutants have lower fitness and are therefore less
likely to contribute any offspring to the population at the time of
fixation. Another step towards a more realistic modeling of genetic
hitchhiking under recurrent mutation would be to allow for beneficial
mutation to the same (physiological) allele at multiple different
positions of the genome. In such a model, recombination between the
different positions of the beneficial mutation in the genome would
complicate our analysis.

\section{Results}
\label{S:results}
The process of fixation of the beneficial allele is described by the
diffusion \eqref{eq:diff}. In Section \ref{sub:fixtime}, we will
derive approximations for the fixation time $T$ of this process. These
results will be needed in Section \ref{sub:yule}, where we construct
an approximation for the structured coalescent $\xi^{\mathcal X}$.

\subsection{Fixation times} 
\label{sub:fixtime}
In the study of the diffusion \eqref{eq:diff} the time $T$ of fixation
of the beneficial allele (see \eqref{eq:defT}) is of particular
interest. We decompose the interval $[0;T]$ by the last time a
frequency of $X_t=0$ was reached, i.e., we define
\begin{equation*}
  \begin{aligned}
    T_0 & := \sup\{t\geq 0: X_t=0\}, \qquad T^\ast := T-T_0.
  \end{aligned}
\end{equation*}
Note that for $\theta\geq 1$, the boundary $x=0$ is inaccessible, such
that $T_0=0, T^\ast=T$, almost surely, in this case. 

\begin{proposition}
\label{P1}
\begin{enumerate}
\item Let $\gamma_e\approx 0.57$ be Euler's $\gamma$. For $\theta>0$,
  \begin{align}
    \label{eq:P1a}
    \mathbb E^0_{\alpha,\theta}[T] & =\frac 1\alpha \Big(
    2\log(2\alpha) + 2\gamma_e + \frac 1\theta - \theta
    \sum_{n=1}^\infty \frac{1}{n(n+\theta)}\Big) + \mathcal O\Big(
    \frac{\log\alpha}{\alpha^2}\Big) + \frac 1\theta \mathcal O\big(
    \alpha e^{-\alpha}\big)
  \end{align}
\item For $\theta\geq 1$, almost surely, $T=T^\ast$.
\item For $0\leq \theta\leq 1$,
  \begin{align}
    \label{eq:P2a}
    \mathbb E^0_{\alpha,\theta}[T^\ast] & = \frac 2\alpha \big(
    \log(2\alpha) + \gamma_e\big) + \mathcal O \Big(
    \frac{\log\alpha}{\alpha^2}\Big)
  \end{align}
\item For $\theta\geq 0$,
  \begin{align}
    \label{eq:P2b}
    \mathbb V^0_{\alpha,\theta}[T^\ast] & = \mathcal O\Big( \frac
    1{\alpha^2}\Big).
  \end{align}
\end{enumerate}

All error terms are in the limit for large $\alpha$ and are uniform on
compacta in $\theta$.
\end{proposition}

\begin{remark}
\mbox{}

\begin{enumerate}
\item Note that \eqref{eq:P1a} reduces to \eqref{eq:P2a} for $\theta =
  1$ as it should since $T_0 \xrightarrow{\theta\uparrow 1} 0$.
\item For $\theta\leq 1$, we find that $\mathbb
  E^0_{\alpha,\theta}[T^\ast]$ is independent of $\theta$ to the order
  considered. In particular it is identical to the conditioned
  fixation time without recurrent mutation ($\theta = 0$) that was
  previously derived
  \citep{vanHerwaarden2002,HermissonPennings2005,EtheridgePfaffelhuberWakolbinger2006}.
  A detailed numerical analysis (not shown) demonstrates that the
  passage times of the beneficial allele decrease at intermediate and
  high frequencies, but increase at low frequencies $X \lesssim
  1/\alpha$ where recurrent mutation prevents the allele from dying
  out.  Both effects do not affect the leading order and precisely
  cancel in the second order for large $\alpha$.
\item To leading order in $1/\theta$ and $\alpha$, the total fixation
  time \eqref{eq:P1a} is
  $$   \mathbb E^0_{\alpha,\theta}[T] \approx \frac{1}{\alpha \theta} +
  E^0_{\alpha,\theta}[T^\ast].$$ Since the fixation probability of a
  new beneficial mutation is $P_{\text{fix}} \approx 2s$ and the rate
  of new beneficial mutations per time unit (of $N$ generations) is
  $N\theta/2$, mutations that are destined for fixation enter the
  population at rate $sN\theta = \alpha\theta$. The total fixation
  time thus approximately decomposes into the conditioned fixation
  time $\mathbb{E}[T^\ast]$ and the exponential waiting time for the
  establishment of the beneficial allele $\frac{1}{\alpha\theta}$.
\item In applications, selective sweeps are found with $\alpha\geq
  100$. We can then ignore the error term $\tfrac 1\theta \mathcal
  O\big( \alpha e^{-\alpha/2}\big)$ in \eqref{eq:P1a} even for
  extremely rare mutations with $\theta \sim 10^{-10}$.
\item The proof of Proposition \ref{P1} can be found in Section
  \ref{sub:P1}.
  \end{enumerate}
\end{remark}

\subsection{The Yule approximation}
\label{sub:yule}
We will provide a useful approximation of the coalescent process with
rates defined in \eqref{eq:rates1a}--\eqref{eq:rates1b}. As already
seen in the last section the process of fixation of the beneficial
allele can be decomposed into two parts. First, the beneficial allele
has to be established, i.e., its frequency must not hit 0 any more.
Second, the established allele must fix in the population. The first
phase has an expected length of about $1/(\alpha\theta)$ and hence may
be long even for large values of $\alpha$, depending on $\theta$.  The
second phase has an expected length of order $(\log\alpha)/\alpha$ and
is thus short for large $\alpha$, independently of $\theta$. For the
potentially long first phase we give an approximation for the
distribution of the coalescent on path space by a finite Kingman
coalescent. For the short second phase, we obtain an approximation of
the distribution of the coalescent (which is started at time $T$) at
time $T_0$ using a Yule process with immigration (which constructs a
genealogy forward in time). To formulate our results, define
$$ 
\beta_0:=T-T_0.
$$ 
Setting $X_t=0$ for $t<0$ we will obtain approximations for the distribution
of coalescent states at time $\beta_0$, 
\begin{equation*}
    \xi_{\beta_0} := (\xi_{\beta_0}^B, \xi_{\beta_0}^b) := \int
    \mathbb P_{\alpha,\theta}[d\mathcal X] \xi_{\beta_0}^{\mathcal X},
\end{equation*}
and of the genealogies for $\beta > \beta_0$, i.e.\ in the phase 
prior to establishment of the beneficial allele, 
\begin{equation*}
    \xi_{\geq \beta_0}:=(\xi_{\geq \beta_0}^B, \xi_{\geq \beta_0}^b) := \int \mathbb
    P_{\alpha,\theta}[d\mathcal X] (\xi_{\beta_0+t}^{\mathcal X})_{t\geq 0}.
\end{equation*}
Note that $\xi_{\beta_0}\in \mathbb{S}_n$ while
$\xi_{\geq \beta_0}\in \mathcal D([0;\infty), \mathbb{S}_n)$, the space of
cadlag paths on $[0;\infty)$ with values in $\mathbb{S}_n$.

~

\begin{figure}
\begin{center}
\includegraphics[width=13.5cm]{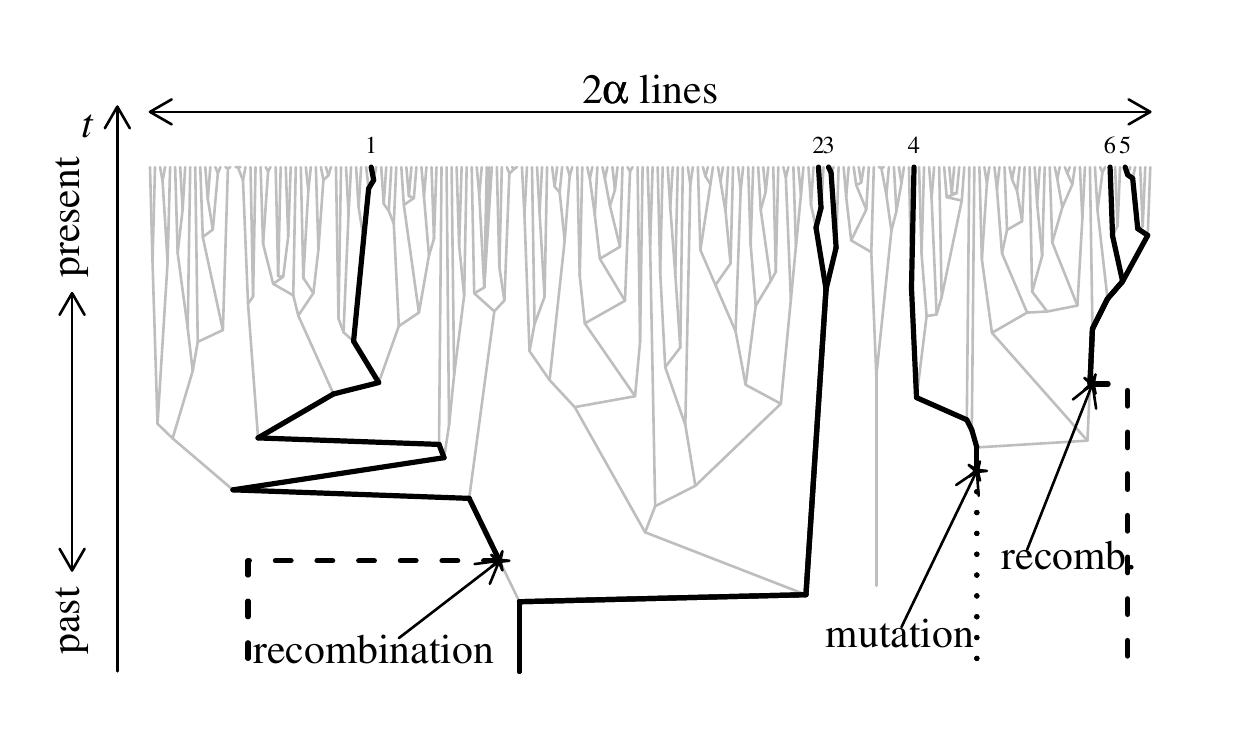}

\vspace{-1cm}
\end{center}
\caption{\label{fig:yule}The Yule process approximation for the
  genealogy at the neutral locus in a sample of size
  $n=6$. {\color{black} The Yule process with immigration produces a
    random forest (grey lines) which grows from the past (past) to the
    present (top). A sample is drawn in the present. Every line is
    marked at constant rate along the Yule forest indicating
    recombination events. Sample individuals within the same tree not
    separated by a recombination mark share ancestry and thus belong
    to the same partition element of $\Upsilon$.} In this realization,
  we find $\Upsilon = \{\{1\}, \{2,3\}, \{4\}, \{5,6\}\}$.
}
\end{figure}

Let us start with $\xi_{\beta_0}$ (see Figure \ref{fig:yule} for an
illustration of our approximation). Consider the selected site first.
Take a Yule process with immigration. Starting with a single line,
\begin{itemize}
\item every line splits at rate $\alpha$.
\item new lines (mutants) immigrate at rate $\alpha\theta$.
\end{itemize}
For this process we speak of Yule-time $i$ for the time the Yule
process has $i$ lines for the first time. We stop this Yule process
with immigration at Yule-time $\lfloor 2\alpha\rfloor$. In order to
define identity by descent within a sample of $n$ lines, take a sample
of $n$ randomly picked lines from the $\lfloor 2\alpha\rfloor$. Note
that the Yule process with immigration defines a random forest
$\mathcal F$ and we may define the random partition
$\widetilde\Upsilon$ of $\{1,...,n\}$ by saying that
$$ k\sim_{\widetilde\Upsilon} \ell \iff \, k,\ell \text{ are in the same tree of }\mathcal F.$$
As a special case of Theorem \ref{T1} we will show that
$\widetilde\Upsilon$ is a good approximation to $\xi_{\beta_0}$ in the
case $\rho=0$.

In order to extend the picture to the general case with recombination,
consider a single line of the neutral allele at time $T$. The line may
recombine in the interval $[T_0,T]$ and thus have an ancestor at time
$T_0$, which carries the wild-type allele. Since recombination events
take place with a rate proportional to $\rho$ and $T-T_0=T^\ast$ is of
the order $(\log\alpha)/\alpha$, it is natural to use the scaling
\begin{align}\label{eq:defGamma}
  \rho = \gamma \frac{\alpha}{\log\alpha}.
\end{align}
Take a sample of $n$ lines from the $\lfloor 2\alpha\rfloor$ lines of
the top of the Yule tree and consider the subtree of the $n$
lines. Indicating recombination events, we mark all branches in the
subtree independently. A branch in the subtree, which starts at
Yule-time $i_1$ and ends at Yule-time $i_2$ is marked with probability
$1-p_{i_1}^{i_2}(\gamma,\theta)$, where
\begin{align}\label{eq:pii}
  p_{i_1}^{i_2}(\gamma,\theta) := \exp\Big(
  -\frac{\gamma}{\log\alpha}\sum_{i=i_1+1}^{i_2} \frac
  1{i+\theta}\Big).
\end{align}
Then, define the random partition $\Upsilon$ of $\{1,...,n\}$ (our
approximation of $\xi_{\beta_0}$) by
$$ k\sim_\Upsilon \ell\iff k\sim_{\widetilde\Upsilon}\ell \;\wedge\;\text{path from }k \text{ to }
\ell \text{ in }\mathcal F\text{ not separated by a mark.}$$

~

To obtain an approximation of $\xi_{\geq\beta_0}$ consider the finite
Kingman coalescent $\mathcal C := ( C_t)_{t\geq 0}$. Given there are
$m$ lines such that $ C_t = C = \{ C_1,..., C_m\}$, transitions occur
for $1\leq 1<j\leq m$ to
$$ \big(  C\setminus \{ C_i, C_j\}\big) 
\cup \{ C_i\cup C_j\}\text{ with rate } 1.$$ Given $\Upsilon$, our
approximation of $\xi_{\geq\beta_0}$ is
$$\Upsilon\circ \mathcal C := (\Upsilon\circ  C_t)_{t\geq 0}.$$

\begin{remark}
  Our approximations are formulated in terms of the total variation
  distance of probability measures. Given two probability measures
  $\mathbb P, \mathbb Q$ on a $\sigma$-algebra $\mathcal A$, the total
  variation distance is given by
  $$ d_{TV}(\mathbb P, \mathbb Q) = \tfrac 12 \sup_{A\in \mathcal A} |\mathbb P[A] - \mathbb Q[A]|.$$
  Similarly, for two random variables $X, Y$ on $\Omega$ with
  $\sigma(X)=\sigma(Y)$ and distributions $\mathcal L(X)$ and
  $\mathcal L(Y)$ we will write
  $$ d_{TV}(X, Y) = d_{TV}(\mathcal L(X), \mathcal L(Y)).$$
\end{remark}

\newpage

\begin{theorem}\label{T1}
  \mbox{}
  \begin{enumerate}
  \item 
The distribution of coalescent states $\xi_{\beta_0}$ at time $\beta_0$ 
under the full model can be approximated by a distribution of coalescent states 
of a Yule process with immigration. In particular,
    \begin{align} \label{eq:alsu}\mathbb
      P_{\alpha,\theta}[\xi_{\beta_0}^B = \emptyset]=1
    \end{align}
    and the bound
    \begin{align}\label{eq:bound1a}
      d_{TV} \big(\xi_{\beta_0}^b, \Upsilon\big) &= \mathcal O \Big(
      \frac{1}{(\log\alpha)^2}\Big)
    \end{align}
    holds in the limit of large $\alpha$ and is uniform on compacta in
    $n, \gamma$ and $\theta$.
  \item The distribution of genealogies $\xi_{\geq\beta_0}$ prior to
    the establishment of the beneficial allele can be approximated by
    the distribution of genealogies under a composition of a Yule
    process with immigration and the Kingman coalescent.  In
    particular,
    \begin{align*}
      \mathbb P[\xi_{\geq\beta_0}^B \neq (\emptyset)_{t\geq 0}] =
      \mathcal O \Big( \frac{1}{\alpha\log\alpha}\Big)
    \end{align*}
    and the bound
    \begin{align}\label{eq:bound2}
      d_{TV} \big(\xi_{\geq\beta_0}^b, \Upsilon \circ \mathcal C\big) & =
      \mathcal O \Big( \frac{1}{(\log\alpha)^2}\Big)
    \end{align}
    holds in the limit of large $\alpha$ and is uniform on compacta
    in $n, \gamma$ and $\theta$.
  \end{enumerate}
\end{theorem}

\begin{remark}
\mbox{}
\begin{enumerate}
\item Let us give an intuitive explanation for the approximation of
  the genealogy at the selected site by $\widetilde\Upsilon$. Consider
  a finite population of size $N$. It is well-known that a
  supercritical branching process is a good approximation for the
  frequency path $\mathcal X$ at times $t$ when $X_t$ is small. In
  such a process, each individual branches at rate 1. It either splits
  in two with probability $\tfrac{1+s}{2}$ or dies with probability
  $\tfrac{1-s}{2}$. In this setting every line has a probability of
  $2s + \mathcal O(s^2)\approx 2\alpha/N$ to be of infinite
  descent. In particular, new mutants that have an infinite line of
  descent arise approximately at rate $2s \cdot Nu =
  \alpha\theta/N$. In addition, when there are $2Ns$ lines of infinite
  descent there must be approximately $N$ lines in total, which is the
  whole population.
\item Using the approximation of $\xi_{\beta_0}$ by $\Upsilon$ we can
  immediately derive a result found in \cite{PenningsHermisson2006b}:
  when the Yule process has $i$ lines the probability that the next
  event (either a split of a Yule line or an incoming mutant) is a
  split is $\frac{i}{\theta+i}$, and that it is an incoming mutant is
  $\frac{\theta}{\theta+i}$. This implies that the random forest
  $\mathcal F$ is generated by Hoppe's urn. Recall also the related
  Chinese restaurant process; see \cite{Aldous1983} and
  \cite{JoyceTavare1987}. The resulting sizes of all families is given
  by the Ewens' Sampling Formula for the $\lfloor 2\alpha\rfloor$
  lines when the Yule tree is stopped. Moreover, the Ewens' Sampling
  Formula is consistent, i.e., subsamples of a large sample again
  follow the formula.
\item When biologists screen the genome of a sample for selective
  sweeps, they can not be sure to have sampled at time $t=T$. Given
  they have sampled lines linked to the beneficial type at $t<T$ when
  the beneficial allele is already in high frequency (e.g. $X_t
  \approx 1 - \delta/\log\alpha$ for some $\delta>0$), the
  approximations of Theorem \ref{T1} still apply. The reason is that
  neither recombination events changing the genetical background nor
  coalescences occur in $[t;T]$ in $\xi$ with high probability; see
  Section \ref{rem7}. If $t>T$, a good approximation to the genealogy
  is $\widetilde{\mathcal C}\circ \Upsilon\circ\mathcal C$ where
  $\widetilde{\mathcal C}$ is a Kingman coalescent run for time $t-T$.
\item The model parameters $n,\gamma$ and $\theta$ enter the error
  terms $\mathcal O(.)$ above. The most severe error in
  \eqref{eq:bound1a} arises from ignoring events with two
  recombination events on a single line. See also Remark
  \ref{rem:5}. Hence, $\gamma$ enters the error term
  quadratically. Since each line might have a double-recombination
  history, the sample size $n$ enters this error term linearly. The
  contribution of $\theta$ to the error term cannot be seen directly
  and is a consequence of the dependence of the frequency path
  $\mathcal X$ on $\theta$.
    
  Note that coalescence events always affect pairs of lines while both
  recombination and mutation affects only single lines.  As a
  consequence, $n$ enters quadratically into higher order error
  terms. In particular, for practical purposes, the Yule process
  approximation becomes worse for big samples.
\item The proof of Theorem \ref{T1} can be found in Section
  \ref{sub:T1}. Key facts needed in the proof are collected in Section
  \ref{sub:key}.
  \end{enumerate}
\end{remark}

\subsection{Application: Expected heterozygosity}
\label{Sub:het} The approximation of Theorem 1 using a Yule forest as
a genealogy has direct consequences for the interpretation of
population genetic data. While genealogical trees cannot be observed
directly, their impact on measures of DNA sequence diversity in a
population sample can be described. The idea is that mutations along
the genealogy of a sample produce polymorphisms that can be
observed. Genealogies in the neighbourhood of a recent adaptation
event are shorter, on average, meaning that sequence diversity is
reduced. This reduction is stronger, however, for a 'hard sweep' (see
Section 2.3), where the sample finds a common ancestor during the time
of the selective phase $\mathbb E[T^\ast]\approx 2\log(\alpha)/\alpha
\ll 1$ than for a 'soft sweep', where the most recent common ancestor
is older. Using our fine asymptotics for genealogies, we are able to
quantify the prediction of sequence diversity under genetic
hitchhiking with recurrent mutation. In this section we will
concentrate on heterozygosity as the simplest measure of sequence
diversity.

By definition, heterozygosity is the probability that two randomly
picked lines from a population are different. Writing $H_t$ for the
heterozygosity at time $t$ and using \eqref{eq:alsu}, we obtain
$$ H_T = \mathbb P_{\alpha,\theta}[\xi_{\beta_0}^b=\{\{1\},\{2\}\}|\xi_0 
= (\{1\},\{2\},\emptyset)] \cdot H_{T_0}.$$ Assuming that the
population was in equilibrium at time $0$, we can use Theorem
\ref{T1}, in particular \eqref{eq:bound1a}, to obtain an approximation
for the heterozygosity at time $T$.

\begin{proposition}
  \label{P3}
  Abbreviating $p_i:=p_i^{\lfloor 2\alpha\rfloor}(\gamma,\theta)$
  (compare \eqref{eq:pii}), heterozygosity at time $T$ is approximated
  by
  \begin{align}\label{eq:P3}
    \frac{H_T}{H_{T_0}} = 1 - \frac{p_{1}^2 }{\theta+1} -
    \frac{2\gamma}{\log\alpha} \sum_{i=2}^{\lfloor 2\alpha\rfloor}
    \frac{2i+\theta}{(i+\theta)^2(i+1+\theta)}p_{i}^2 +
    \mathcal O\Big( \frac{1}{(\log\alpha)^2}\Big)
  \end{align}
  where the error is in the limit of large $\alpha$ and is uniform on
  compacta in $n, \gamma$ and $\theta$.
\end{proposition}

\begin{figure}
\begin{center}
\includegraphics[width=13.5cm]{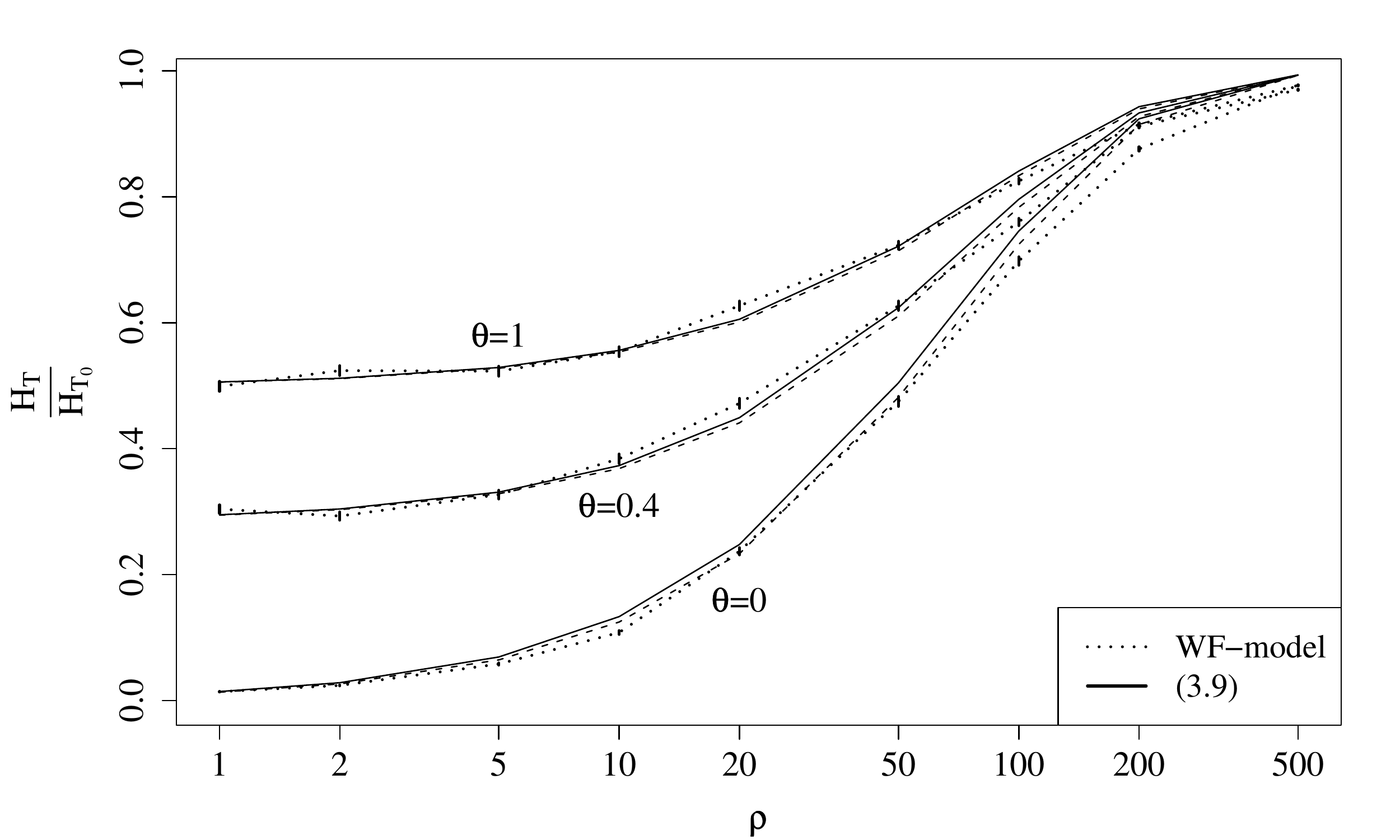}
\end{center}
\caption{\label{fig:sim} Reduction in heterozygosity at time of
  fixation of the beneficial allele. The $x$-axis shows the
  recombination distance of the selected from the neutral locus. Solid
  lines connect results from the analytical approximation. Dotted
  lines show simulation results of a structured coalescent in a
  Wright-Fisher model with $N = 10^4$ and $\alpha = 1000$. Small
  vertical bars indicate standard errors from $10^3$ numerical
  iterations.}
\end{figure}

\begin{remark}
\mbox{}

  \begin{enumerate}
  \item The formula \eqref{eq:P3} establishes that
    \begin{align}\label{eq:P3a}
      \frac{H_T}{H_{T_0}} = 1 - \frac{p_{1}^2 }{\theta+1} + \mathcal
      O\Big( \frac{1}{\log\alpha}\Big).
    \end{align}
    In particular, to a first approximation, two lines taken from the
    population at time $T$ are identical by descent if their linked
    selected locus has the same origin (probability $\frac
    1{1+\theta}$) and if both lines were not hit by independent
    recombination events (probability $p_1^2$).
  \item We investigated the quality of the approximation \eqref{eq:P3}
    by numerical simulations. The outcome can be seen in Figure
    \ref{fig:sim}. As we see, for $\alpha=1000$, our approximation
    works well for all values of $\theta\leq 1$ up to
    $\rho/\alpha=0.1$, i.e., $\gamma=0.7$.
  \item We can compare Proposition \ref{P3} with the result for the
    heterozygosity under a star-like approximation for the genealogy
    at the selected site, which was used by
    \citet[eq. (8)]{PenningsHermisson2006b}, i.e.
    \begin{align}\label{eq:P3PH}
      \frac{H_T}{H_{T_0}} \approx 1 - \frac{e^{-2\gamma} }{\theta+1}.
    \end{align}
    Note that this formula also arises approximately by taking
    $p_1^{\lfloor 2\alpha\rfloor}(\gamma,0)$ instead of $p_1$ in
    \eqref{eq:P3a}.  As shown in Table \ref{tab:num}, the additional
    terms from the Yule process approximation lead to an improvement
    over the simple star-like approximation result.

    \begin{table}
      \begin{center}
        \begin{tabular}{|c|cccc|} \hline & $\theta= 0,\rho=2$ &
          $\theta= 0,\rho=5$ & $\theta=0, \rho=10$ & $\theta=
          0,\rho=50$ \\\hline
          WF-model &0.024&0.058&0.108&0.475\\
          \eqref{eq:P3} &0.028(17\%)&0.069(19\%)&0.133(23\%)&0.504(6\%)\\
          \eqref{eq:P3PH} &0.032(33\%)&0.079(36\%)&0.151(40\%)&0.559(18\%)\\
          \hline \hline & $\theta= 0.1,\rho=2$ & $\theta= 0.1,\rho=5$
          & $\theta=0.1, \rho=10$ & $\theta= 0.1,\rho=50$ \\\hline
          WF-model &0.112&0.153&0.223&0.507\\
          \eqref{eq:P3} &0.116(4\%)&0.152(1\%)&0.209(6\%)&0.541(7\%)\\
          \eqref{eq:P3PH} &0.12(7\%)&0.162(6\%)&0.228(2\%)&0.599(18\%)\\
          \hline \hline & $\theta= 1,\rho=2$ & $\theta= 1,\rho=5$ &
          $\theta=1, \rho=10$ & $\theta= 1,\rho=50$ \\\hline
          WF-model &0.524&0.523&0.554&0.723\\
          \eqref{eq:P3} &0.512(2\%)&0.529(1\%)&0.556(0\%)&0.722(0\%)\\
          \eqref{eq:P3PH} &0.516(2\%)&0.539(3\%)&0.575(4\%)&0.779(8\%)\\
          \hline
        \end{tabular}
      \end{center}
      \caption{\label{tab:num}Comparison of numerical simulation of a Wright-Fisher 
        model to \eqref{eq:P3} and \eqref{eq:P3PH}. Numbers in brackets are the 
        relative error of the approximation. For $\theta=0$ and $\theta=1$, 
        the same set of simulations as in Figure 
        \ref{fig:sim} are used. In particular, $N=10^4$ and $\alpha=1000$.}
    \end{table}

  \item The quantification of sequence diversity patterns for
    selective sweeps with recurrent mutation using the Yule process
    approximation is not restricted to heterozygosity. Properties of
    several other statistics could be computed. As an example, we
    mention the site frequency spectrum, which describes the number of
    singleton, doubleton, tripleton, etc, mutations in the sample.

    Moreover, as pointed out by Pennings and Hermisson (2006b),
    selective sweeps with recurrent mutation also lead to a distinct
    haplotype pattern around the selected site. Intuitively, every
    beneficial mutant at the selected site brings along its own
    genetic background leading to several extended
    haplotypes. Quantifying such haplotypes patterns would require
    models for more than one neutral locus.
  \end{enumerate}
\end{remark}

\begin{proof}[Proof of Proposition \ref{P3}]
  Using Theorem \ref{T1} we have to establish that $\mathbb
  P_{\alpha,\theta}[\xi_{\beta_0}^b=\{\{1,2\}\}|\xi_0 =
  (\{1\},\{2\},\emptyset)]$ is approximately given by the right hand
  side of \eqref{P3}. To see this, we compute, accounting for all
  possibilities when coalescence of two lines can occur,
\begin{align*}
  \mathbb P[\Upsilon=\{\{1,2\}\}] & = \sum_{i=1}^{\lfloor
    2\alpha\rfloor} \frac{i}{i+\theta}\frac{1}{\binom{i+1}{2}} p_i^2
  \cdot \prod_{j=i+1}^{\lfloor 2\alpha\rfloor }
  \Big(\frac{\theta}{j+\theta} \Big( 1 - \frac{2}{j+1}\Big) +
  \frac{j}{j+\theta} \Big( 1 - \frac{1}{\binom{j+1}{2}}\Big)\Big) \\ &
  = \sum_{i=1}^{\lfloor 2\alpha\rfloor }
  \frac{2p_i^2}{(i+\theta)(i+1)} \cdot \prod_{j=i+1}^{\lfloor
    2\alpha\rfloor } \Big( \theta \frac{j-1}{(j+\theta)(j+1)} +
  \frac{(j-1)(j+2)}{(j+\theta)(j+1)}\Big)\\
  & = \sum_{i=1}^{\lfloor 2\alpha\rfloor }
  \frac{2p_i^2}{(i+\theta)(i+1)} \cdot \prod_{j=i+1}^{\lfloor
    2\alpha\rfloor } \frac{j-1}{j+1}\frac{j+2+\theta}{j+\theta} \\ & =
  \sum_{i=1}^{\lfloor 2\alpha\lfloor } \frac{2p_i^2}{i+\theta}
  \frac{i}{(i+1+\theta)(i+2+\theta)} + \sO\Big(\frac 1 \alpha\Big)
  \\
  & = \sum_{i=1}^{\lfloor 2\alpha\rfloor} \Big(
  \frac{2}{(i+1+\theta)(i+2+\theta)} -
  \frac{2\theta}{(i+\theta)(i+1+\theta)(i+2+\theta)}\Big) p_{i}^2 +
  \sO\Big(\frac 1\alpha\Big).
\end{align*}
Rewriting gives
\begin{align*}
  \mathbb P[\Upsilon=\{\{1,2\}\}] & = \sum_{i=1}^{\lfloor
    2\alpha\rfloor } \Big( \frac{2p_{i}^2}{i+1+\theta} -
  \frac{2p_{i}^2}{i+2+\theta}\Big) \\ & \qquad \qquad \qquad \qquad -
  \theta\sum_{i=1}^{\lfloor 2\alpha\rfloor } \Big(
  \frac{p_{i}^2}{(i+\theta)(i+1+\theta)}-
  \frac{p_{i}^2}{(i+1+\theta)(i+2+\theta)} \Big) + \sO\Big( \frac
  1\alpha \Big) \\ & = \frac{2p_{1}^2}{\theta+2} + \sum_{i=1}^{\lfloor
    2\alpha\rfloor } \frac{2(p_{i+1}^2-p_{i}^2)}{i+2+\theta} \\ &
  \qquad \qquad \qquad \qquad -
  \theta\frac{p_{1}^2}{(\theta+1)(\theta+2)} -
  \theta\sum_{i=1}^{\lfloor 2\alpha \rfloor}
  \frac{p_{i+1}^2-p_{i}^2}{(i+1+\theta)(i+2+\theta)}
  + \sO\Big(\frac 1\alpha\Big) \\
  & = \frac{p_{1}^2}{\theta+1} + \sum_{i=1}^{\lfloor 2\alpha\rfloor }
  \frac{2i+\theta+2}{(i+1+\theta)(i+2+\theta)}(p_{i+1}^2-p_{i}^2) +
  \sO\Big(\frac 1\alpha\Big) \\ & = \frac{p_{1}^2}{\theta+1} +
  \frac{2\gamma}{\log\alpha} \sum_{i=2}^{\lfloor 2\alpha\rfloor}
  p_i^2\frac{2i+\theta}{(i+\theta)^2(i+1+\theta)} + \sO\Big(\frac
  1{(\log\alpha)^2}\Big)
\end{align*}
where the last equality follows from
$$ p_{i+1}^2-p_{i}^2 = p_{i+1}^2
\Big( 1 - \exp\Big(-\frac{2\gamma}{\log\alpha}\frac 1{i+1+\theta}
\Big)\Big) = p_{i+1}^2 \frac{2\gamma}{\log\alpha}\frac 1{i+1+\theta} +
\frac 1{i^2}\sO\Big(\frac{1}{(\log\alpha)^2}\Big).$$
\end{proof}

\section{Proof of Proposition \ref{P1} (Fixation times)}
\label{sub:P1}
Our calculations are based on the Green function $t(.;.)$ for the
diffusion $\mathcal X=(X_t)_{t\geq 0}$. This function satisfies
\begin{align}\label{eq:green1}
  \mathbb E_{\alpha,\theta}^p\Big[ \int_{0}^T f(X_t) dt\Big] =
  \int_0^1 t(x;p)f(x) dx
\end{align}
and
\begin{align}
  \label{eq:green2}
  \mathbb E_{\alpha,\theta}^p\Big[ \int_{0}^T \int_{t}^T f(X_t) g(X_s)
  ds dt\Big] = \int_0^1 \int_0^1 t(x;p) t(y;x) f(x) g(y) dy dx.
\end{align}
Using
\begin{equation*}
  \begin{aligned}
    \psi_{\alpha,\theta}(y) := \psi(y) := \exp\Big( -2\int_1^y
    \frac{\mu_{\alpha,\theta}(z)}{\sigma^2(z)} dz \Big) =
    \frac{1}{y^{\theta}} \exp(2\alpha(1-y))
  \end{aligned}
\end{equation*}
the Green function for $\mathcal X$, started in $p$, is given by
(compare \cite{Ewens2004}, (4.40), (4.41))
\begin{equation*}
  \begin{aligned}
    t_{\alpha,\theta}(x;p) & = \frac{2}{\sigma^2(x)\psi(x)}\int_{x\vee
      p}^1 \psi(y) dy  =
    \frac{2}{x(1-x)} \int_{x\vee p}^1
    e^{-2\alpha (y-x)} \Big(\frac{x}{y}\Big)^{\theta}dy.
  \end{aligned}
\end{equation*}
Since $T^\ast$ depends only on the path conditioned not to hit 0, we
need the Green function of the conditioned diffusion. To derive its
infinitesimal characteristics, we need the absorption probability,
i.e., given a current frequency of $p$ of the beneficial allele, its
probability of absorption at 1 before hitting 0. This probability is
given by
\begin{equation*}
  \begin{aligned}
    P^1_{\alpha,\theta}(p) & = \frac{\int_0^p \psi(y) dy}{\int_0^1
      \psi(y)dy} = \frac{\int_0^p \frac{e^{-2\alpha
          y}}{y^{\theta}}dy}{\int_0^1 \frac{e^{-2\alpha
          y}}{y^{\theta}}dy}
  \end{aligned}
\end{equation*}
for $\theta<1$. For $\theta\geq 1$, we have $P^1_{\alpha,\theta}=1$,
i.e., 0 is an inaccessible boundary. In the case $\theta<1$, the Green
function of the conditioned process is for $p\leq x$ (compare
\cite{Ewens2004}, (4.50))
\begin{equation*}
  \begin{aligned}
    t^\ast_{\alpha,\theta}(x;p) = P^1_{\alpha,\theta}(x) \cdot
    t_{\alpha,\theta}(x;p)
  \end{aligned}
\end{equation*}
and for $x\leq p$ (see \cite{Ewens2004}, (4.49))
\begin{equation*}
  \begin{aligned}
    t^\ast_{\alpha,\theta}(x;p) & = \frac{2}{\sigma^2(x)\psi(x)}
    \frac{(1-P^1_{\alpha,\theta}(x))P^1_{\alpha,\theta}(x)}{P^1_{\alpha,\theta}(p)}
    \int_0^x \psi(y) dy \\ & = 2 \frac{1}{\sigma^2(x)\psi(x)}
    \frac{\int_p^1 \psi(y) dy \int_0^x \psi(y)dy \int_0^x \psi(y)
      dy}{\int_0^p \psi(y) dy \int_0^1 \psi(y) dy}.
  \end{aligned}
\end{equation*}


~

Before we prove Proposition \ref{P1} we give some useful estimates.

\begin{lemma}
\label{L1a}
\begin{enumerate}
\item
For $\varepsilon,K\in (0;\infty)$ there exists $C\in\mathbb R$ such that
\begin{align}\label{L1a:1a}
  \sup_{\varepsilon\leq x\leq 1, 0\leq \theta\leq K}
  \Big|\frac{1-x^\theta}{\theta(1-x)}\Big| \leq C.
\end{align} 
\item For $\theta\in[0;1)$,
\begin{equation}\label{eq:Gamm1}
  \begin{aligned}
    \int_0^1 z^{-\theta} e^{-2\alpha z}dz =
    \frac{1}{2\alpha^{1-\theta}} \Gamma(1-\theta) + \mathcal
    O(e^{-2\alpha})
  \end{aligned}
\end{equation}
where $\Gamma(.)$ is the Gamma function.
\item The bounds
\begin{align}
  \label{L1a:3b}
  \int_0^1 x^{\theta-1} e^{-2\alpha(1-x)} dx & = \mathcal O\Big( \frac
  1\alpha\Big) + \frac 1\theta \mathcal O\big( \alpha e^{-\alpha}\big) ,\\
  \label{L1a:3a}
  \int_0^1 \frac{1-e^{-2\alpha x}}{x} dx - \log2\alpha + \gamma_e & =
  \mathcal O\Big(\frac 1\alpha\Big), \\
  \label{L1a:3c}
  \int_0^1 \frac{1-x^\theta}{1-x} e^{-2\alpha(1-x)} dx & = \mathcal O
  \Big( \frac 1\alpha\Big),\\
  \label{L1a:3d}
  \int_0^1 \int_0^{y/2} \frac 1{1-x} \Big( \frac xy\Big)^\theta
  e^{-2\alpha(y-x)} dx dy &= \mathcal O\Big( \frac{1}{\alpha^2}\Big)
\end{align}
hold in the limit of large $\alpha$, and uniformly on compacta in
$\theta$.
\end{enumerate}

\end{lemma}
\begin{proof}
\begin{enumerate}
\item By a Taylor approximation of $x\mapsto x^\theta$ around $x=1$ we
  obtain
$$ x^\theta = 1 + \theta(1-x) + \tfrac{\theta(\theta-1)}{2} \xi^{\theta-2}(1-x)^2$$
for some $x\leq \xi\leq 1$ and the result follows.
\item We simply compute
  \begin{equation}\label{eq:log1}
    \begin{aligned}
      \int_0^1 z^{-\theta} e^{-2\alpha z}dz =
      \frac{1}{(2\alpha)^{1-\theta}} \int_0^{2\alpha} e^{-z}
      z^{-\theta} dz = \frac{1}{(2\alpha)^{1-\theta}} \Gamma(1-\theta)
      + \mathcal O(e^{-2\alpha})
    \end{aligned}
  \end{equation}
\item For \eqref{L1a:3b}, we write
  \begin{equation*}
    \begin{aligned}
      \int_0^1 x^{\theta-1} e^{-2\alpha(1-x)} dx &= \frac 1\theta
      x^\theta e^{-2\alpha(1-x)}\Big|_0^1 +
      \frac{2\alpha}\theta\int_0^1 x^\theta e^{-2\alpha(1-x)} dx \\ &
      = \frac 1\theta - \frac {2\alpha}\theta \int_0^1 e^{-2\alpha x}
      dx + \mathcal O\Big( \tfrac {2\alpha}\theta e^{-\alpha} +
      2\alpha \int_0^1 (1-x) e^{-2\alpha(1-x)} dx\Big) \\ & = \mathcal
      O\Big( \frac 1\alpha\Big) + \frac 1 \theta \mathcal O\big(
      \alpha e^{-\alpha}\big)
    \end{aligned}
  \end{equation*}
  where we have used 1. for $\varepsilon=\tfrac 12$. For
  \eqref{L1a:3a}, see \cite[p. 61]{Bronstein1982}.  Equation
  \eqref{L1a:3c} follows from
  \begin{align*}
    \int_0^1 \frac{1-x^\theta}{1-x} e^{-2\alpha(1-x)} dx & \leq
    \int_0^1 \frac{1-x^{\lceil \theta\rceil}}{1-x} e^{-2\alpha(1-x)}
    dx = \sum_{i=0}^{\lceil \theta \rceil} \int_0^1 x^i
    e^{-2\alpha(1-x)}dx \leq \frac{\lceil\theta\rceil}{2\alpha}
  \end{align*}
  and \eqref{L1a:3d} from 
  \begin{align*}
    \int_0^1 \int_0^{y/2} \frac 1{1-x} \Big( \frac xy\Big)^\theta
    e^{-2\alpha(y-x)} dx dy & \leq 2\int_0^1\int_0^{y/2}
    e^{-2\alpha(y-x)}dx dy \\ & = \frac{1}{2\alpha} \int_0^1
    e^{-2\alpha y} - e^{-2\alpha y/2} dy = \mathcal O\Big(
    \frac{1}{\alpha^2}\Big).
  \end{align*}
\end{enumerate}
\end{proof}

\begin{lemma}
  \label{l:fixP}
  Let $2\alpha\geq 1$. There is $C>0$ such that for all
  $\theta\in[0;1]$ and $x\in[0;1]$
  $$ P^1_{\alpha,\theta}(x) \leq \big( C(2\alpha x)^{1-\theta}\big) \wedge 1. $$
\end{lemma}

\begin{proof}
  By a direct calculation, we find
  \begin{align*}
    P^1_{\alpha,\theta}(x) & = \frac{\int_0^{2\alpha x}
      \frac{e^{-y}}{y^{\theta}}dy}{\int_0^{2\alpha}
      \frac{e^{-y}}{y^{\theta}}dy} \leq \frac{\int_0^{2\alpha x}
      y^{-\theta}}{\int_0^1 \frac{e^{-1}}{y^\theta}} = e \cdot
    (2\alpha x)^{1-\theta}
  \end{align*}
  Moreover, since $P^1_{\alpha,\theta}(x)$ is a probability, the bound
  $P^1_{\alpha,\theta}(x)\leq 1$ is obvious.
\end{proof}

\begin{proof}[Proof of Proposition \ref{P1}]
  We start with the proof of \eqref{eq:P1a}, i.e., we set $f=1$ in
  \eqref{eq:green1}.  We split the integral of $\mathbb
  E_{\alpha,\theta}[T]$ by using $\tfrac 1{x(1-x)} = \tfrac 1x +
  \tfrac 1{1-x}$, i.e.,
\begin{equation*}
\begin{aligned}
  \mathbb E_{\alpha,\theta}[T] & = 2\int_0^1 \int_0^y \frac 1x \Big(
  \frac xy\Big)^\theta e^{-2\alpha(y-x)} dx dy + 2\int_0^1 \int_0^y
  \frac 1{1-x} \Big( \frac xy\Big)^\theta e^{-2\alpha(y-x)} dx dy.
\end{aligned}
\end{equation*}
For the first part,
\begin{align*}
  2\int_0^1 & \int_0^y \frac 1x \Big( \frac xy\Big)^\theta
  e^{-2\alpha(y-x)} dx dy \stackrel{x\to x/y}{=} 2\int_0^1 \int_0^1
  x^{\theta-1} e^{-2\alpha y(1-x)} dy dx \\ & = \frac 1{\alpha} \int_0^1
  x^{\theta-1} \frac 1{1-x} \big( 1 - e^{-{2\alpha}(1-x)}\big) dx \\ & =
  \frac 1{\alpha} \int_0^1 \Big( x^{\theta-1} +
  \frac{x^\theta}{1-x}\Big)\big( 1 - e^{-{2\alpha}(1-x)}\big) dx \\ &
  =\frac1{\alpha}\Big( \frac 1\theta - \int_0^1
  \frac{1-x^\theta}{1-x}(1-e^{-{2\alpha}(1-x)})dx + \int_0^1
  \frac{1-e^{-{2\alpha}(1-x)}}{1-x}dx \Big) + \mathcal O\Big( \frac
  1{\alpha^2}\Big) + \frac 1\theta \mathcal O\big( \alpha
  e^{-\alpha}\big)\\ & = \frac 1{\alpha}\Big( \frac 1\theta -
  \sum_{n=0}^\infty \int_0^1 (x^n - x^{n+\theta})dx + \log({2\alpha}) +
  \gamma_e \Big) + \mathcal O\Big( \frac 1{\alpha^2}\Big) + \frac
  1\theta \mathcal O\big( {2\alpha} e^{-{\alpha}}\big)\\ & = \frac
  1{\alpha}\Big( \frac 1\theta - \theta \sum_{n=1}^\infty \frac{1}{n( n
    + \theta )} + \log({2\alpha}) + \gamma_e \Big) + \mathcal O\Big( \frac
  1{\alpha^2}\Big)+ \frac 1\theta \mathcal O\big( \alpha
  e^{-{\alpha}}\big).
\end{align*}
where we have used \eqref{L1a:3b} in the fourth and both,
\eqref{L1a:3a} and \eqref{L1a:3c} in the fifth equality.  The second
part gives, using \eqref{L1a:3d} and \eqref{L1a:1a},
\begin{equation}\label{eq:pp411}
\begin{aligned}
  2\int_0^1 \int_0^y & \frac 1{1-x} \Big( \frac xy\Big)^\theta
  e^{-{2\alpha}(y-x)} dx dy = 2\int_0^1 \int_0^{y/2} \frac 1{1-x}
  \Big( \frac xy\Big)^\theta e^{-{2\alpha}(y-x)} dx dy \\ & \qquad
  \qquad \qquad \qquad \qquad \qquad \qquad \qquad + 2\int_0^1
  \int_0^{y/2} \frac{1}{1-y+x}\Big( 1 - \frac xy \Big)^\theta
  e^{-{2\alpha} x} dx dy \\ & = 2\int_0^1 \int_0^y \frac 1{1-y+x}
  e^{-{2\alpha} x} dx dy + \mathcal O\Big( \int_0^1 \int_0^y
  \frac{x}{1+x} \Big( \frac{1}{1-y+x} + \frac{1}{y}\Big) e^{-{2\alpha}
    x} dx dy\Big) + \mathcal O\Big( \frac{1}{\alpha^2}\Big) \\ & =
  2\int_0^1 \int_0^y \frac{1}{1-x} e^{-{2\alpha}(y-x)} dy dx +
  \mathcal O\Big( \int_0^1 x \log\Big(
  \frac{y}{1-y+x}\Big)\Big|_{y=x}^{y=1} e^{-{2\alpha} x} dx\Big) +
  \mathcal O\Big( \frac{1}{\alpha^2}\Big) \\ & = \frac 1{\alpha}
  \int_0^1 \frac{1-e^{-{2\alpha}(1-x)}}{1-x}dx + \mathcal O\Big(
  \frac{1}{\alpha^2} \int_0^{2\alpha} x \log\Big( \frac {2\alpha}
  x\Big) e^{-x} dx\Big) + \mathcal O\Big( \frac{1}{\alpha^2}\Big) \\ &
  = \frac 1{\alpha} \big( \log 2 \alpha + \gamma_e\big) + \mathcal O
  \Big(\frac {\log\alpha}{\alpha^2}\Big)
\end{aligned}
\end{equation}
and \eqref{eq:P1a} follows. For the proof of \eqref{eq:P2a} we have
\begin{equation}\label{eq:pp412a}
  \begin{aligned}
    \mathbb E_{\alpha,\theta}[T^\ast] = \frac{2 \int_0^1 \int_x^1
      \int_0^x
      \frac{e^{-2\alpha(y-x)}}{x(1-x)}\Big(\frac{x}{yz}\Big)^\theta
      e^{-2\alpha z} dz dy dx}{\int_0^1 z^{-\theta} e^{-2\alpha z}dz}
  \end{aligned}
\end{equation}
By using $\tfrac 1{x(1-x)} = \tfrac 1x + \tfrac 1{1-x}$ we again split
the integral in the numerator. For the $\tfrac 1x$-part we find
\begin{equation*}
  \begin{aligned}
    2 \int_0^1 \int_x^1 \int_0^x & \frac{e^{-2\alpha(y-x)}}{x} \Big(
    \frac{x}{yz}\Big)^\theta e^{-2\alpha z} dz dy dx \stackrel{x\to
      x/y}{=} 2\int_0^1\int_0^1 \int_0^{xy}
    \frac{e^{-2\alpha(y(1-x))}}{x} \Big( \frac xz\Big)^\theta
    e^{-2\alpha z} dz dy dx \\ & \stackrel{z\to
      z/x}{=}2\int_0^1\int_0^1
    \int_0^y e^{-2\alpha y(1-x)} z^{-\theta} e^{-2\alpha zx} dz dy dx \\
    & = \frac 1\alpha \int_0^1\int_0^1 \frac{1}{1-x} \big( e^{-2\alpha
      z(1-x)} - e^{-2\alpha(1-x)}\big) z^{-\theta} e^{-2\alpha zx} dz
    dx \\ & \stackrel{1\to 1-x}{=} \frac 1\alpha \int_0^1 \int_0^1
    \frac 1x \big( e^{-2\alpha z} - e^{-2\alpha (x+z-xz)}\big)
    z^{-\theta} dz dx \\ & = \frac 1\alpha \int_0^1
    \frac{1-e^{-2\alpha x}}{x} dx \int_0^1 e^{-2\alpha z} z^{-\theta}
    dz + \frac 1\alpha \int_0^1 \int_0^1 \frac{1}{x} e^{-2\alpha
      (x+z)}\big(1 - e^{2\alpha xz}\big) z^{-\theta} dx dz.
  \end{aligned}
\end{equation*}
Using \eqref{eq:Gamm1} we see that 
\begin{equation}\label{eq:pp413}
  \begin{aligned}
    \int_0^1 \int_0^1 & \frac 1x e^{-2\alpha(x+z)}\big( e^{2\alpha xz}
    -1\big) z^{-\theta} dx dz = \sum_{n=1}^\infty \int_0^1 e^{-2\alpha
      z} z^{n-\theta} dz \int_0^{2\alpha} e^{-x} \frac{x^{n-1}}{n!} dx \\
    & = \mathcal O\Big( \int_0^1 e^{-2\alpha z}
    z^{-\theta}\sum_{n=1}^\infty \frac{z^n}{n} dz\Big) \\ & = \mathcal
    O\Big( \int_0^1 e^{-2\alpha z} z^{-\theta}\log(1-z) dz\Big) \\ & =
    \mathcal O\Big( \int_0^{1} z^{1-\theta} e^{-2\alpha z} dz\Big) \\
    & = \Gamma(2-\theta) \mathcal O\Big(
    \frac{1}{\alpha^{2-\theta}}\Big)
  \end{aligned}
\end{equation}
such that, with \eqref{eq:Gamm1},
\begin{equation}\label{eq:pp415}
  \begin{aligned}
    \frac{2\int_0^1 \int_x^1 \int_0^x
      \frac{e^{-\alpha(y-x)}}{x}\Big(\frac{x}{yz}\Big)^\theta
      e^{-\alpha z} dz dy dx}{\int_0^1 z^{-\theta} e^{-\alpha z}dz} =
    \frac 1\alpha \big( \log2\alpha + \gamma_e\big) + \mathcal O\Big(
    \frac{1}{\alpha^2}\Big).
  \end{aligned}
\end{equation}
For the $\tfrac{1}{1-x}$-part, we write 
\begin{equation*}
  \begin{aligned}
    \Big| \int_0^1 \int_x^1 \int_0^x &
    \frac{e^{-2\alpha(y-x)}}{1-x}\Big(\frac{x}{yz}\Big)^\theta
    e^{-2\alpha z} dz dy dx - \Big( \int_0^1 z^{-\theta} e^{-2\alpha
      z}dz\Big) \Big( \int_0^1 \int_x^1 \frac {e^{-2\alpha(y-x)}}{1-x}
    \Big( \frac xy\Big)^\theta dy dx\Big)\Big| \\ & = \int_0^1
    \int_x^1 \int_x^1 \frac{e^{-2\alpha(y-x)}}{1-x} \Big(
    \frac{x}{yz}\Big)^\theta e^{-2\alpha z} dzdydx \\ & = \mathcal
    O\Big( \int_0^1 z^{-\theta} e^{-2\alpha z} \int_0^z \int_x^1
    \frac{e^{-2\alpha(y-x)}}{1-x} dy dx dz\Big) \\ & = \mathcal O\Big(
    \frac 1\alpha \int_0^1 z^{-\theta} e^{-2\alpha z} \int_{1-z}^1
    \frac{1-e^{-2\alpha x}}{x} dx dz \Big) \\ & = \mathcal O\Big(
    \frac 1\alpha\int_0^1 z^{-\theta} e^{-2\alpha z} \log(1-z)dz \Big)
    \\ & = \mathcal O \Big( \frac 1\alpha \int_0^1
    z^{1-\theta}e^{-2\alpha z} dz \Big) \\ & = \Gamma(2-\theta)
    \mathcal O\Big( \frac{1}{\alpha^{3-\theta}} \Big)
\end{aligned}
\end{equation*}
where we have used \eqref{eq:Gamm1} in the last step. Hence, by
\eqref{eq:pp411},
\begin{equation}\label{eq:pp417}
  \begin{aligned}
    \frac{2\int_0^1 \int_x^1 \int_0^x
      \frac{e^{-2\alpha(y-x)}}{1-x}\Big(\frac{x}{yz}\Big)^\theta
      e^{-2\alpha z} dz dy dx}{\int_0^1 z^{-\theta} e^{-2\alpha z}dz}
    = \frac 1\alpha \big( \log 2\alpha + \gamma_e\big) + \theta
    \mathcal O\Big( \frac{\log\alpha}{\alpha^2}\Big) + \mathcal
    O\Big(\frac 1{\alpha^2}\Big).
  \end{aligned}
\end{equation}
Plugging \eqref{eq:pp415} and \eqref{eq:pp417} into \eqref{eq:pp412a}
gives \eqref{eq:P2a}.

For the variance we start with $\theta<1$. By \eqref{eq:green2} and a
similar calculation as in the proof of Lemma \ref{l:fixP}, for some
finite $C$ (which is independent of $\theta$ and $\alpha$), using
\eqref{eq:green2} 
\begin{equation}
  \label{eq:42a0}
  \begin{aligned}
    \mathbb V^0& [T^\ast] = 2\int_0^1 \int_0^w t^\ast_\theta (w;0)
    t^\ast_\theta (x;w) dx dw \\ & = 4 \int_0^1 \int_0^w
    \frac{e^{2\alpha(w+x)}}{w^{1-\theta}(1-w)x^{1-\theta}(1-x)}
    \Big(\int_w^1 \frac{e^{-2\alpha y}}{y^\theta} dy\Big)
    \Big(\int_w^1 \frac{e^{-2\alpha z}}{z^\theta} dz\Big) \Big(
    \frac{\int_0^x \frac{e^{-2\alpha \hat z}}{\hat z^\theta}d\hat
      z}{\int_0^1 \frac{e^{-2\alpha \tilde z}}{\tilde z^\theta}d\tilde
      z}\Big)^2 dx dw \\ &
    \stackrell{w,x,y,z\to}{\leq}{2\alpha(w,x,y,z)} \frac{C}{\alpha^2}
    \int_0^{2\alpha} \int_0^w \int_w^1 \int_w^1
    \frac{e^{w+x-y-z}}{w(1-\frac w{2\alpha})x(1-\frac x{2\alpha})}
    \Big( \frac{wx}{yz}\Big)^\theta (x^{2-2\theta}\wedge 1) dz dy dx
    dw \\ & = \frac{2C}{\alpha^2} \int_0^{2\alpha} \int_0^z \int_0^y
    \int_0^w \Big( \frac 1w + \frac{1}{2\alpha-w}\Big) \Big( \frac 1x +
    \frac{1}{2\alpha-x}\Big)e^{w+x-y-z} \Big( \frac{wx}{yz}\Big)^\theta
    (x^{2-2\theta}\wedge 1) dx dw dy dz
  \end{aligned}
\end{equation}
where the last equality follows by the symmetry of the integrand with
respect to $y$ and $z$. We divide the last integral into several
parts. Moreover, we use that 
\begin{align}\label{eq:trick}
  \int_0^{2\alpha} \int_0^z \int_0^y \int_0^w ... dx dw dy dz =
  \int_0^{2\alpha} \int_0^z \int_0^y \int_0^{w\wedge 1} ... dx dw dy
  dz + \int_1^{2\alpha} \int_1^z \int_1^y \int_1^w ... dx dw dy dz.
\end{align}
First,
\begin{equation}
  \label{eq:42aa}
  \begin{aligned}
    \int_0^{2\alpha} & \int_0^z \int_0^y \int_0^w \frac 1{wx}
    e^{w+x-y-z} \Big( \frac{wx}{yz}\Big)^\theta (x^{2-2\theta}\wedge
    1) dx dw dy dz \\ & = \mathcal O\Big( \int_0^\infty \int_0^z
    \int_0^y \int_0^{w\wedge 1} \frac{e^{-z}}{wx}\Big(
    \frac{wx}{yz}\Big)^\theta x^{2-2\theta} dx dw dy dz +
    \int_1^\infty \int_1^z \int_1^y \int_1^w \frac{e^{w+x-y-z}}{wx} dx
    dw dy dz\Big) \\ & = \mathcal O\Big(
    \frac{1}{2-\theta}\int_0^\infty \int_0^z \int_0^y e^{-z} w \Big(
    \frac{1}{yz}\Big)^\theta dw dy dz + \int_1^\infty \int_x^\infty
    \int_w^\infty \int_y^\infty \frac{e^{w+x-y-z}}{wx} dz dy dw
    dx\Big) \\ & = \mathcal O\Big( \int_0^\infty \int_0^z
    \frac{e^{-z}}{z^\theta} y^{2-\theta} dy dz + \int_1^\infty
    \int_x^\infty \frac{e^{x-w}}{wx} dw dx\Big) \\ & = \mathcal O\Big(
    \int_0^\infty z^{3-2\theta} e^{-z}dz + \int_1^\infty \frac{1}{x^2}
    dx\Big) = \mathcal O(1)
  \end{aligned}
\end{equation}
Second, since $\frac{1}{w(2\alpha-x)}\leq \frac{1}{x(2\alpha-w)}$ for
$x\leq w$,
\begin{equation}
  \label{eq:42ab}
  \begin{aligned}
    \int_0^{2\alpha} & \int_0^z \int_0^y \int_0^w \Big( \frac
    1{w(2\alpha-x)} + \frac{1}{x(2\alpha-w)}\Big) e^{w+x-y-z} \Big(
    \frac{wx}{yz}\Big)^\theta (x^{2-2\theta}\wedge 1) dx dw dy dz \\ &
    = \mathcal O\Big( \int_0^\infty \int_0^z \int_0^y \int_0^{w\wedge
      1} \frac{e^{-z}}{x(2\alpha-w)} \Big(\frac{wx}{yz}\Big)^\theta
    x^{2-2\theta} dx dw dy dz \\ & \qquad \qquad \qquad \qquad \qquad
    \qquad \qquad \qquad + \int_1^{2\alpha} \int_1^z \int_1^y \int_1^w
    \frac{e^{w+x-y-z}}{x({2\alpha}-w)} dx dw dy dz\Big) \\ & =
    \mathcal O\Big( \int_0^{2\alpha} \int_0^z \int_0^y
    \underbrace{\frac{w^2}{{2\alpha}-w}}_{\leq
      \frac{({2\alpha})^2}{{2\alpha}-w}-{2\alpha}} e^{-z} \Big(
    \frac{1}{yz}\Big)^\theta dw dy dz + \int_1^{2\alpha}
    \int_x^{2\alpha} \int_w^{2\alpha} \int_y^\infty
    \frac{e^{w+x-y-z}}{x({2\alpha}-w)} dz dy dw dx\Big) \\ & =
    \mathcal O\Big( \alpha^2 \int_0^\infty \int_0^z \big(
    -\log(1-\tfrac y{2\alpha}) - \tfrac y {2\alpha}\big) e^{-z} \Big(
    \frac{1}{yz}\Big)^\theta dy dz + \int_1^{2\alpha} \int_x^{2\alpha}
    \frac{e^{w+x}(e^{-2w} - e^{-4\alpha})}{x({2\alpha}-w)} dw dx\Big) \\
    & = \mathcal O \Big( \int_0^\infty \int_0^z
    y^{2-\theta}\frac{e^{-z}}{z^\theta} dy dz + \int_1^{2\alpha}
    \int_0^{{2\alpha}-x} \frac{e^{-{2\alpha}+x}(e^w-e^{-w})}{wx} dw dx\Big) \\
    & = \mathcal O\Big( \int_0^\infty z^{3-2\theta}e^{-z}dz +
    \int_1^{{2\alpha}-1} \underbrace{\frac{1}{x({2\alpha}-x)}}_{=
      \frac 1{2\alpha}\big( \frac 1x + \frac{1}{{2\alpha}-x}\big)} dx
    \Big) + \mathcal O(1) \\ & = \mathcal O(1).
  \end{aligned}
\end{equation}
Third, 
\begin{equation}
  \label{eq:42ac}
  \begin{aligned}
    \int_0^{2\alpha} & \int_0^z \int_0^y \int_0^w \frac
    1{({2\alpha}-w)({2\alpha}-x)} e^{w+x-y-z} \Big(
    \frac{wx}{yz}\Big)^\theta (x^{2-2\theta}\wedge 1) dx dw dy dz \\ &
    \stackrell{(w,x,y,z)\to}{=}{{2\alpha}-(w,x,y,z)}\mathcal O\Big(
    \int_0^{2\alpha} \int_z^{2\alpha}\int_y^{2\alpha}\int_w^\infty
    \frac{e^{y+z-x-w}}{wx}dx dwdydz\Big) \\ & = \mathcal O\Big(
    \int_0^{2\alpha} \int_0^x\int_0^w\int_0^y \frac{e^{y+z-x-w}}{wx}
    dz dy dw dx\Big) \\ & = \mathcal O\Big( \int_0^{2\alpha} \int_0^x
    \frac{e^{-x}(e^{w}-e^{-w})}{xw}dw dx\Big) \\ & = \mathcal O\Big(
    \int_0^{2\alpha} \int_0^{x\wedge 1} \frac{e^{-x}}{x} dw dx +
    \int_1^{2\alpha} \int_1^x \frac{e^{w-x}}{xw} dwdx \Big) = \mathcal
    O(1).
  \end{aligned}
\end{equation}
Plugging \eqref{eq:42aa}, \eqref{eq:42ab}, \eqref{eq:42ac} into
\eqref{eq:42a0} gives \eqref{eq:P2b} for $\theta<1$.

For $\theta \geq 1$, we compute
\begin{equation}\label{eq:Var1}
\begin{aligned}
  \mathbb V_{\alpha,\theta}[T^\ast] & = \mathbb V_{\alpha,\theta}[T] =
  2\int_0^1 \int_0^w \int_w^1 \int_w^1
  \frac{e^{-{2\alpha}(w+x-y-z)}}{w(1-w)x(1-x)} \Big(
  \frac{wx}{yz}\Big)^\theta dz dy dx dw \\ & \leq 4 \int_0^1 \int_0^z
  \int_0^y \int_0^w \frac{e^{-{2\alpha}(w+x-y-z)}}{y(1-w)z(1-x)} dx dw
  dy dz \\ & \stackrell{(w,x,y,z)\to}{=}{{2\alpha}(w,x,y,z)} \mathcal
  O\Big( \int_0^{2\alpha} \int_0^z \int_0^y \int_0^{w\wedge 1}
  \frac{e^{-z}}{y({2\alpha}-w)z({2\alpha}-x)} dx dw dy dz \\ & \qquad
  \qquad \qquad \qquad \qquad \qquad + \int_1^{2\alpha}
  \int_x^{2\alpha} \int_w^{2\alpha} \int_y^\infty
  \frac{e^{w+x-y-z}}{({2\alpha}-w)({2\alpha}-x)yz} dz dy dw dx\Big) \\
  & = \mathcal O\Big( \frac 1\alpha \int_0^{2\alpha} \int_0^z \int_0^y
  \frac{e^{-z}}{yz} \frac{w\wedge 1}{{2\alpha}-w} dw dy dz +
  \int_1^{2\alpha} \int_x^{2\alpha} \int_w^{2\alpha}
  \frac{e^{w+x-2y}}{({2\alpha}-w)({2\alpha}-x)y^2} dy dw dx \Big) \\ &
  = \mathcal O \Big( \frac{1}{\alpha^2} \int_0^{2\alpha} \int_0^z
  \frac{e^{-z}}{z} \log\big(1-\tfrac y{2\alpha}\big) dy dz +
  \int_1^{{2\alpha}-1}
  \int_x^{{2\alpha}-1}\frac{e^{x-w}}{({2\alpha}-w)({2\alpha}-x)w^2} dw
  dx + \frac 1{\alpha^2} \Big) \\ & = \mathcal O\Big(
  \frac{1}{\alpha^2} \int_0^{2\alpha} e^{-z} \log\big( 1 - \tfrac
  z{2\alpha}\big) dz + \int_1^{{2\alpha}-1}
  \Big(\frac{1}{({2\alpha}-x)x}\Big)^2 dx + \frac{1}{\alpha^2} \Big)
  \\ & = \mathcal O\Big( \frac{1}{\alpha^2} +\frac{1}{\alpha^2}
  \int_1^{{2\alpha}-1}\Big( \frac 1x + \frac{1}{{2\alpha}-x}\Big)^2 dx
  \Big) = \mathcal O \Big( \frac{1}{\alpha^2}\Big).
\end{aligned}
\end{equation}
\end{proof}

\section{Key Lemmata}
\label{sub:key}
In this section we prove some key facts for the proof of Theorem
\ref{T1}. Recall $\rho = \gamma \frac{\log\alpha}{\alpha}$ from
\eqref{eq:defGamma} and let $\xi_1^{\mathcal X}, \xi_2^{\mathcal X},
\xi_3^{\mathcal X}, \xi_4^{\mathcal X}, \xi_5^{\mathcal X}$ and
$\xi_6^{\mathcal X}$ be Poisson-processes conditioned on $\mathcal X$
with rates $\tfrac 1{X_t}, \frac\theta 2\frac{1-X_t}{X_t}, \rho
(1-X_t), \rho X_t, 1$ and $\frac{X_t}{1-X_t}$, at time $t$, as given
in Table \ref{tab:1}. Moreover, let $T_i^{\mathcal X} := \sup \xi_i^{\mathcal
  X}$ be the last event of $\xi_i^{\mathcal X}$, $i=1,...,4$ in
$[0;T]$.

{\color{black} Note that $\xi_1^{\mathcal X}$ give the pair coalescence
  rates in $B$. In addition, coalescences in the wild-type background
  might happen due to events in $\xi_5^{\mathcal X} \cup
  \xi_6^{\mathcal X}$ since $1 + \frac{X_t}{1-X_t} =
  \frac{1}{1-X_t}$. The other processes determine changes in the
  genetic background due to mutation ($\xi_2^{\mathcal X}$) and
  recombination ($\xi_3^{\mathcal X}, \xi_4^{\mathcal X}$).}

\begin{table}
  \begin{center}
\begin{tabular}{|c|cccccc|}\hline
  \rule[-3mm]{0cm}{1cm}process & \hspace{2ex}$\xi_1^{\mathcal X}$\hspace{2ex} & \hspace{2ex}$\xi_2^{\mathcal X}$\hspace{2ex} & \hspace{2ex}$\xi_3^{\mathcal X}$\hspace{2ex} & \hspace{2ex}$\xi_4^{\mathcal X}$\hspace{2ex} & \hspace{2ex}$\xi_5^{\mathcal X}$\hspace{2ex} & \hspace{2ex}$\xi_6^{\mathcal X}$\hspace{2ex} \\[2ex]\hline
  \rule[-3mm]{0cm}{1cm}rate & $\tfrac{1}{X_t}$ & $\frac \theta 2\frac{1-X_t}{X_t}$ & $\rho(1-X_t)$ &  $\rho X_t$ & $1$ & $\frac{X_t}{1-X_t}$\\\hline
  \rule[-5mm]{0cm}{1.2cm} interpretation & \parbox{1.9cm}{coalescence in $B$} & \parbox{1.9cm}{mutation from $B$ to $b$} & \parbox{1.9cm}{recombination from $B$ to $b$} & \parbox{1.9cm}{recombination from $b$ to $B$} & \multicolumn{2}{c|}{\parbox{1.9cm}{coalescence in $b$}} \\\hline
\end{tabular}
\end{center}
\caption{\label{tab:1}Rates of Poisson processes}
\end{table}

We will prove three Lemmata. The first deals with events of the Poisson
processes during $[0;T_0]$. Recall that $T_0>0$ iff $\theta < 1$. The
second lemma is central for \eqref{eq:alsu}, i.e., to prove that no
lines are in the beneficial background at time $T_0$. The third Lemma
helps to order events during $[T_0;T]$. We use the convention that
$[s;t]=\emptyset$ for $s>t$.

\begin{lemma}
  \label{l:est1}
  Let $\theta<1$. Then,
  \begin{align}
    \label{eq:l:est1:1}
    \mathbb P^0_{\alpha,\theta}[\xi_4^{\mathcal X} \cap[0;T_0]\neq
    \emptyset] & = \mathcal O\Big(
    \frac{1}{\alpha\log\alpha}\Big),\\
    \label{eq:l:est1:2}
    \mathbb P^0_{\alpha,\theta}[\xi_6^{\mathcal X} \cap[0;T_0]\neq
    \emptyset] & = \mathcal O\Big( \frac{1}{\alpha^2}\Big).
  \end{align}
  All error terms are in the limit for large $\alpha$, are uniform in
  $\theta$ and uniform on compacta in $\gamma$.
\end{lemma}

\begin{lemma}
  \label{l:est1a}
  For all values of $\theta$ and $\alpha$, 
  \begin{align*}
    \mathbb P^0_{\alpha,\beta}[ \xi_2^{\mathcal X}\cap [T_0; T_0+) =\emptyset] = 0.
  \end{align*}
\end{lemma}

\begin{lemma}
  \label{l:est2}
  The bounds
  \begin{align}
    \label{eq:l:est2:1}
    \mathbb P^0_{\alpha,\theta}[\xi_4^{\mathcal X}
    \cap[T_0;T_2^{\mathcal X}]\neq \emptyset] & = \mathcal O\Big(
    \frac{1}{\sqrt\alpha}\Big),\\
    \label{eq:l:est2:1b}
    \mathbb P^0_{\alpha,\theta}[\xi_4^{\mathcal X} \cap[T_0;T_3^{\mathcal X}]\neq
    \emptyset] & = \mathcal O\Big(
    \frac{1}{(\log\alpha)^2}\Big),\\
    \label{eq:l:est2:2}
    \mathbb P^0_{\alpha,\theta}[\xi_5^{\mathcal X} \cap[T_0;T]\neq
    \emptyset] & = \mathcal O\Big( \frac{\log\alpha}{\alpha}\Big),\\
    \label{eq:l:est2:3}
    \mathbb P^0_{\alpha,\theta}[\xi_6^{\mathcal X}
    \cap[T_0;T_2^{\mathcal X}]\neq \emptyset] & =
    \mathcal O\Big( \frac{1}{\sqrt\alpha}\Big),\\
    \label{eq:l:est2:3b}
    \mathbb P^0_{\alpha,\theta}[\xi_6^{\mathcal X}
    \cap[T_0;T_3^{\mathcal X}]\neq \emptyset] & = \mathcal O\Big(
    \frac{\log\alpha}{\alpha}\Big).
  \end{align}
  hold in the limit for large $\alpha$, are uniform on compacta in
  $\theta$ and $\gamma$.
\end{lemma}

\begin{remark}\label{rem:5}
  Lemmata \ref{l:est1} and \ref{l:est2} are crucial in ordering events
  in $\xi^{\mathcal X}$ (recall all rates from Table \ref{tab:2}). In
  particular, let us consider events in $[T_0;T]$, i.e., the bounds
  from Lemma \ref{l:est2}. The full argument for the application of
  Lemmata \ref{l:est1}-\ref{l:est2} is given in the proof of Theorem
  \ref{T1} in Section \ref{sub:T1}.  Consider a single line (i.e. a
  sample of size 1).  {\color{black}Recall from Table \ref{tab:1} that
    the} processes $\xi_i^{\mathcal X}, i=2,3,4$ determine
  changes in the genetic background due to mutation ($\xi_2^{\mathcal
    X}$) and recombination ($\xi_3^{\mathcal X}, \xi_4^{\mathcal
    X}$).

  As we see from \eqref{eq:l:est2:1b}, the event that the line
  (backwards in time) changes background by recombination to the
  wild-type and back to the beneficial background has a probability of
  order $\mathcal O \big( \frac{1}{(\log\alpha)^2}\big)$. The event
  that the line changes genetic background by mutation and recombines
  back to the beneficial background has a probability of order
  $\mathcal O \big( \frac{1}{\sqrt\alpha}\big)$ by
  \eqref{eq:l:est2:1}.  The event of a coalescence in the wild-type
  background requires that both lines change background to the
  wild-type and so, necessarily, one event from \eqref{eq:l:est2:2},
  \eqref{eq:l:est2:3} or \eqref{eq:l:est2:3b} must take place. Hence,
  the probability of a coalescence event in the wild-type background
  is of the order $\mathcal O \big( \frac{1}{\sqrt\alpha}\big)$.
\end{remark}

\begin{proof}[Proof of Lemma \ref{l:est1}]
  For \eqref{eq:l:est1:1}, since $P^1_\theta$ is monotone increasing
  in $\theta$ and $P_{\alpha,0}^1(p) = \frac{1-e^{-2\alpha
      p}}{1-e^{-2\alpha}}$, we compute, using Lemma \ref{l:fixP} and
  $\rho=\mathcal O\big( \frac{\alpha}{\log\alpha}\big)$,
  \begin{align*}
    \mathbb P^0_{\alpha,\theta}[\xi_4^{\mathcal X} \cap[0;T_0]\neq
    \emptyset] & = \mathbb E^0_{\alpha,\theta} \Big[ 1 - \exp\Big( -
    \int_0^{T_0} \rho X_t dt
    \Big)\Big] \leq \mathbb E^0_{\alpha,\theta} \Big[ \int_0^{T_0} \rho X_t dt\Big] \\
    & = \rho \int_0^1 \big( t_{\alpha,\theta}(x;0) -
    t^\ast_{\alpha,\theta}(x;0)\big) x dx \\ & \leq \rho \int_0^1
    (1-P^1_{\alpha,0}(x))x t_{\alpha,\theta}(x;0) dx \\ & = \mathcal
    O\Big( \rho \int_0^1 \int_x^1 \frac{e^{-2\alpha y}}{1-x}
    \underbrace{\Big( \frac{x}{y} \Big)^{\theta}}_{\leq 1} dy dx \Big)
    \\ & = \mathcal O\Big( \frac{\rho}{\alpha} e^{-2\alpha} \int_0^1
    \frac{e^{2\alpha(1-x)}-1}{1-x}dx\Big) \\ & = \mathcal
    O\Big(\frac{e^{-2\alpha} }{\log\alpha} \int_0^{2\alpha}
    \frac{e^x-1}{x} dx \Big) \\ & = \mathcal O\Big(\frac{1}{\alpha
      \log\alpha}\Big).
  \end{align*}
  For $\xi_6^{\mathcal X}$, by a similar calculation,
  \begin{align*}
    \mathbb P^0_{\alpha,\theta}[\xi_6 \cap[0;T_0]\neq \emptyset] &
    \leq \int_0^1 \big( t_{\alpha,\theta}(x;0) -
    t^\ast_{\alpha,\theta}(x;0)\big) \frac{x}{1-x} dx \\ & = \mathcal
    O\Big( \int_0^1 \int_x^1 \frac{e^{-2\alpha y}(1-e^{-2\alpha(1-x)})
    }{(1-x)^2} \Big( \frac{x}{y} \Big)^{\theta} dy dx\Big) \\ & \leq
    \mathcal O\Big( \frac 1\alpha e^{-2\alpha} \int_0^1
    \frac{(e^{2\alpha (1-x)}-1)(1-e^{-2\alpha(1-x)})}{(1-x)^2} dx
    \Big) \\
    & =
    \mathcal O\Big( e^{-2\alpha} \int_0^{2\alpha} \frac{(e^x-1)(1-e^{-x})}{x^2}dx \Big) \\
    & = \mathcal O\Big( \frac 1{\alpha^2}\Big).
  \end{align*}
\end{proof}

\begin{proof}[Proof of Lemma \ref{l:est1a}]
  Note that the process $\mathcal X$ as well as its time-reversion
  $\mathcal Z=(Z_t)_{t\geq 0}$ with $Z_t:=X_{T-t}$ are special cases
  of the diffusion studied in \cite{Taylor2007}.  We use Lemma 2.1 of
  that paper, which extends Lemma 4.4 of
  \cite{BartonEtheridgeSturm2004}. Their Lemma 2.1 shows that, for all
  $0\leq s \leq T^\ast$,
  $$ \mathbb P_{\alpha,\theta}\Big[ \int_s^{T^\ast} \frac {1-Z_t}{Z_t}dt = \infty \Big] = 1. $$
  In particular,
  $$ \mathbb P_{\alpha,\theta}[\xi_2^{\mathcal X} \cap [T_0, T_0+s) =\emptyset] = 
  \mathbb E_{\alpha,\theta}\Big[ \exp\Big( - \int_{T_0}^{T_0+s} \frac
  \theta 2\frac{1-X_t}{X_t} dt\Big)\Big] = 0.$$ Hence the result
  follows for $s\to 0$.
\end{proof}

\begin{proof}[Proof of Lemma \ref{l:est2}]
  \noindent\emph{Proof of \eqref{eq:l:est2:1}:} Set $\mathcal Y = (Y_t)_{0\leq
    t\leq T^\ast}$ with $Y_t = X_{T-t}$, i.e. $\mathcal Y$ is the
  time-reversion of $(X_{T_0+t})_{0\leq t\leq T^\ast}$. Recall that
  the Green function of the time-reversed diffusion $\mathcal Y$ is
  given for $x\leq p$ by (see \cite{Ewens2004}, (4.51))
  \begin{equation*}
    \label{eq:green3}
    \begin{aligned}
      t^{\ast\ast}(x;p) = 2 \frac{1}{\sigma^2(x) \psi(x)}
      \frac{\int_x^1 \psi(y)dy \int_0^x \psi(y) dy}{\int_0^1 \psi(y)
        dy}
    \end{aligned}
  \end{equation*}
  and for $p\leq x$ by (see \cite{Ewens2004}, (4.52))
  \begin{equation*}
    \label{eq:green3a}
    \begin{aligned}
      t^{\ast\ast}(x;p) = 2 \frac{1}{\sigma^2(x) \psi(x)}
      \frac{\int_x^1 \psi(y)dy \int_0^p \psi(y) dy \int_x^1 \psi(y)
        dy}{\int_p^1 \psi(y) dy \int_0^1 \psi(y) dy}
    \end{aligned}
  \end{equation*}
  with the convention that $\frac{\int_{0}^x\psi(y) dy}{\int_0^1
    \psi(y)dy}=1$ for $\theta\geq 1$. Denote by
  \begin{align*} \widetilde T_{\varepsilon}^{\mathcal X} :=
    \sup\{{t\leq T: X_t=\varepsilon}\} = T - \inf\{t\geq 0:
    Y_t=\varepsilon\}.
  \end{align*}
  We will use
  \begin{align}
    \label{eq:pr1g}
    \mathbb P_{\alpha,\theta}^0[\xi_4^{\mathcal X} \cap [T_0;
    T_2^{\mathcal X}] \neq \emptyset] \leq \mathbb
    P_{\alpha,\theta}^0[\xi_4^{\mathcal X} \cap [T_0; \widetilde
    T_\varepsilon^{\mathcal X}] \neq \emptyset] + \mathbb
    P_{\alpha,\theta}^0[\widetilde T^{\mathcal X}_\varepsilon\leq
    T_2^{\mathcal X} ]
  \end{align}
  and bound both terms on the right hand side separately for
  $\varepsilon = \varepsilon(\alpha) =
  \frac{\log\alpha}{\sqrt\alpha}$.  The bound of the first term is
  established by
  \begin{align*}
    \frac{\int_x^1 \frac{e^{-2\alpha
          y}}{y^\theta}dy}{\int_\varepsilon^1 \frac{e^{-2\alpha
          y}}{y^\theta}dy} = \mathcal O\Big(
    \Big(\frac{\varepsilon}{x}\Big)^\theta
    e^{-2\alpha(x-\varepsilon)}\Big)
  \end{align*}
  uniformly for $\varepsilon\leq x\leq 1$ and
  \begin{align*}
    \mathbb P_{\alpha,\theta}^0 & [\xi_4^{\mathcal X} \cap [T_0;
    \widetilde T_\varepsilon^{\mathcal X}] \neq \emptyset] = \mathbb
    E^0_{\alpha,\theta}\Big[ 1 - \exp\Big( - \rho \int_0^{\widetilde
      T_\varepsilon^{\mathcal X}} X_t dt \Big)\Big] \\ & = \mathbb
    E^\varepsilon_{\alpha,\theta}\Big[ 1 - \exp\Big( - \rho
    \int_0^{\infty} Y_t dt \Big)\Big] \leq \rho \int_0^1
    t^{\ast\ast}_{\alpha,\theta}(x;\varepsilon) x dx \\ & = \mathcal
    O\Big( \rho \int_0^\varepsilon \int_x^1 \frac{1}{1-x} \Big(\frac
    xy\Big)^\theta e^{-2\alpha(y-x)} dy dx + \rho \int_\varepsilon^1
    \int_x^1 \frac{1}{1-x}\Big( \frac\varepsilon y\Big)^\theta
    e^{-2\alpha(y-\varepsilon)}dy dx\Big) \\ & = \mathcal O\Big( \rho
    \int_0^\varepsilon \int_x^1 e^{-2\alpha(y-x)} dy dx + \rho
    \int_\varepsilon^1 \int_\varepsilon^y \frac{1}{1-x}
    e^{-2\alpha(y-\varepsilon)} dx dy\Big) \\ & = \mathcal O\Big(
    \frac{\rho}{\alpha} \int_0^\varepsilon dx + \rho\int_\varepsilon^1
    \log\big( 1 - (y-\varepsilon)\big) e^{-2\alpha(y-\varepsilon)}
    dy\Big) \\ & = \mathcal O\Big( \frac{1}{\sqrt\alpha} +
    \rho\int_0^1 y e^{-2\alpha y}dy\Big) = \mathcal O\Big(
    \frac{1}{\sqrt\alpha} \Big),
  \end{align*}
  while the bound of the second term follows from
  \begin{align}\label{eq:bound1}
    \mathbb P_{\alpha,\theta}^0[\widetilde T^{\mathcal
      X}_\varepsilon\leq T_2^{\mathcal X} ] = \mathbb
    E^0_{\alpha,\theta} \Big[ 1- \exp\Big( - \int_{\widetilde
      T^{\mathcal X}_\varepsilon}^{T} \frac \theta 2
    \frac{1-X_t}{X_t}dt\Big) \Big] \leq \frac \theta 2
    \frac{1}{\varepsilon} \mathbb E^0_{\alpha,\theta}[T^\ast] =
    \mathcal O \Big( \frac{1}{\sqrt\alpha}\Big).
  \end{align}
  Hence, we have bounded both terms on the right hand side of
  \eqref{eq:pr1g} and thus have proved \eqref{eq:l:est2:1}.

  \noindent\emph{Proof of \eqref{eq:l:est2:1b}:}
  Note that by \eqref{eq:green2}
  \begin{equation}
    \label{eq:proof1a0}
    \begin{aligned}
      \mathbb P^0_{\alpha,\theta}& [\xi_4^{\mathcal X}
      \cap[T_0;T_3^{\mathcal X}]\neq \emptyset] \\ & = \mathbb
      E^0_{\alpha,\theta} \Big[ \int_{T_0}^T \Big( 1 - \exp\Big( -
      \int_{T_0}^t \rho X_s ds \Big) \Big) \rho(1-X_t) \exp\Big(
      -\int_t^T \rho(1-X_s)ds \Big) dt \Big] \\ & \leq \rho^2 \mathbb
      E_{\alpha,\theta}^0\Big[ \int_{T_0}^T (1-X_t) \int_{T_0}^tX_s ds
      dt \Big] \\ & = \rho^2 \int_0^1 \int_0^1 t_\theta^\ast(x;0)
      t_\theta^\ast(y;x) x(1-y) dy dx.
    \end{aligned}
  \end{equation}
  We split the last double integral into parts. First,
  \begin{equation}
    \label{eq:proof1a}
    \begin{aligned}
      \int_0^1 \int_0^x t_\theta^\ast(x;0) t_{\alpha,\theta}^\ast(y;x)
      x(1-y) dy dx \leq \mathbb V^0_{\alpha,\theta}[T^\ast] = \mathcal
      O\Big(\frac{1}{\alpha^2}\Big)
    \end{aligned}
  \end{equation}
  by Proposition \ref{P1}, \eqref{eq:P2b}.  Second, recall
  $t_{\alpha,\theta}^\ast(y,x) = t_{\alpha,\theta}^\ast(y;0)$ for
  $x\leq y$. So we have, for all values of $\theta$,
  \begin{equation}
    \label{eq:proof1}
    \begin{aligned}
      \int_0^1 \int_x^1 & t_{\alpha,\theta}^\ast(x;0)
      t_{\alpha,\theta}^\ast(y;0) x(1-y) dy dx \leq \int_0^1 \int_x^1
      t_{\alpha,\theta}(x;0) t_{\alpha,\theta}(y;0) x(1-y) dy dx \\ &
      = \mathcal O \Big( \int_0^1 \int_x^1 \int_x^1 \int_y^1
      e^{-2\alpha(z+z'-x-y)} \Big( \frac{xy}{zz'}\Big)^\theta
      \frac{1}{1-x}\frac 1y dz' dz dy dx \Big) \\ & = \mathcal O \Big(
      \frac{1}{\alpha^2}\int_0^{2\alpha} \int_x^{2\alpha}
      \int_x^{2\alpha} \int_y^{2\alpha} e^{-(z+z'-x-y)}
      \underbrace{\Big( \frac{xy}{zz'}\Big)^\theta}_{\leq 1}
      \frac{1}{2\alpha-x}\frac 1y dz' dz dy dx \Big) \\ & = \mathcal O
      \Big( \frac{1}{\alpha^2} \int_0^{2\alpha} \int_0^{y}
      \frac{1}{2\alpha-x}\frac{1}{y} dx dy\Big) = \mathcal O \Big(
      \frac{1}{\alpha^2}
      \int_0^{2\alpha} \frac{\log(1-\tfrac{y}{2\alpha})}{y} dy\Big) \\
      & = \mathcal O \Big( \frac{1}{\alpha^2}\int_0^1
      \frac{\log(1-y)}{y} dy\Big) = \mathcal O\Big(\frac
      1{\alpha^2}\Big).
    \end{aligned}
  \end{equation}
  Hence, plugging \eqref{eq:proof1a} and \eqref{eq:proof1} into
  \eqref{eq:proof1a0} establishes \eqref{eq:l:est2:1b} since $\rho =
  \mathcal O\Big(\frac{\alpha}{\log\alpha}\Big)$.

  \noindent \emph{Proof of \eqref{eq:l:est2:2}:} We simply observe,
  using Proposition \ref{P1},
  \begin{align*}
    \mathbb P^0_{\alpha,\theta}[\xi_4^{\mathcal X} \cap[T_0;T]\neq
    \emptyset] & = \mathbb E_{\alpha,\theta}^0[1-e^{-T^\ast}] \leq
    \mathbb E^0_{\alpha,\theta}[T^\ast] = \mathcal
    O\Big(\frac{\log\alpha}{\alpha}\Big).
  \end{align*}

  \noindent \emph{Proof of \eqref{eq:l:est2:3}:} We will use the
  time-reversed process $\mathcal Y$ as in the proof of
  \eqref{eq:l:est2:1}. Note that
  \begin{align}
    \label{eq:pr1h}
    \mathbb P_{\alpha,\theta}^0[\xi_6^{\mathcal X} \cap [T_0;
    T_2^{\mathcal X}] \neq \emptyset] \leq \mathbb
    P_{\alpha,\theta}^0[\xi_6^{\mathcal X} \cap [T_0; \widetilde
    T_\varepsilon^{\mathcal X}] \neq \emptyset] + \mathbb
    P_{\alpha,\theta}^0[\widetilde T^{\mathcal X}_\varepsilon\leq
    T_2^{\mathcal X} ]
  \end{align}
  and the last term is bounded by \eqref{eq:bound1}. The first term is
  bounded using
  \begin{align*}
    \frac{1}{\int_\varepsilon^1 \frac{e^{-2\alpha y}}{y^\theta} dy} =
    \mathcal O\big(\alpha \varepsilon^\theta
    e^{2\alpha\varepsilon}\big)
  \end{align*}
  by (recall $\varepsilon=\varepsilon(\alpha) =
  \frac{\log\alpha}{\sqrt\alpha}$)
  \begin{align*}
    \mathbb P_{\alpha,\theta}^0 & [\xi_6^{\mathcal X} \cap [T_0;
    \widetilde T_\varepsilon^{\mathcal X}] \neq \emptyset] \leq
    \int_0^1 t^{\ast\ast}_{\alpha,\theta}(x;\varepsilon) \frac{x}{1-x}
    dx \\ & = \mathcal O\Big( \int_0^\varepsilon \int_x^1
    \frac{1}{(1-x)^2} \Big(\frac xy\Big)^\theta e^{-2\alpha(y-x)} dy
    dx \\ & \qquad \qquad \qquad \qquad \qquad + \alpha
    \int_\varepsilon^1 \int_x^1 \int_x^1 \frac{1}{(1-x)^2}\Big(
    \frac{x\varepsilon}{yz}\Big)^\theta
    e^{-2\alpha(y+z-x-\varepsilon)}dz dy dx\Big) \\ & = \mathcal
    O\Big( \int_0^\varepsilon \int_x^1 e^{-2\alpha(y-x)} dy dx +
    \frac{1}{\alpha} \int_\varepsilon^1 e^{-2\alpha(x-\varepsilon)} dx
    \Big) \\ & = \mathcal O \Big( \frac{\log\alpha}{\alpha^{3/2}} +
    \frac 1{\alpha^2}\Big) = \mathcal O \Big(
    \frac{\log\alpha}{\alpha^{3/2}}\Big).
  \end{align*}
  \noindent\emph{Proof of \eqref{eq:l:est2:3b}:} Note that
  \begin{align*}
    \mathbb P^0[\xi_6^{\mathcal X} \cap[T_0;T_3^{\mathcal X}]\neq
    \emptyset] & \leq \rho \int_0^1 \int_0^1 t_\theta^\ast(w;0)
    t_\theta^\ast(x;w) \frac{w}{1-w} (1-x) dx dw.
  \end{align*}
  We split the last integral and use that $t^\ast(x;w) = t^\ast(x;0)$
  for $w\leq x$, such that
  \begin{align*}
    \int_0^1 \int_w^1 t_{\alpha,\theta}^\ast(w;0)
    t_{\alpha,\theta}^\ast(x;0) \frac{w}{1-w} (1-x) dx dw & \leq
    \mathcal O\Big(\Big( \int_0^1
    t_{\alpha,\theta}^\ast(w;0)dw\Big)^2\Big) \\ & = \mathcal O \big(
    (\mathbb E_{\alpha,\theta}^0[T^\ast])^2\big) = \mathcal
    O\Big(\frac{(\log\alpha)^2}{\alpha^2}\Big)
  \end{align*}
  by Proposition \ref{P1}. For the second part, using
  \eqref{eq:trick}, we have in the case $\theta<1$, by a calculation
  similar to \eqref{eq:42ab},
  \begin{equation}\label{eq:sim1a}
    \begin{aligned}
      \int_0^1 \int_0^w & t_{\alpha,\theta}^\ast(w;0)
      t_{\alpha,\theta}^\ast(x;w) \frac{w}{1-w}(1-x) dx dw \\ & =
      \mathcal O \Big( \int_0^1 \int_0^w \int_w^1 \int_w^1
      \frac{e^{2\alpha(w+x-y-z)}}{(1-w)^2 x} \Big(
      \frac{wx}{yz}\Big)^\theta (2\alpha x \wedge 1)^{2-2\theta} dz dy
      dx dw\Big) \\ & = \mathcal O\Big( \frac 1\alpha \int_0^{2\alpha}
      \int_0^z \int_0^y \int_0^w \frac{e^{w+x-y-z}}{x(2\alpha-w)^2}
      \Big(\frac{wx}{yz}\Big)^\theta (x\wedge 1)^{2-2\theta} dx dw dy
      dz \Big) \\ & = \mathcal O \Big( \frac 1\alpha
      \int_0^{2\alpha}\int_0^z \int_0^y \int_0^{w\wedge 1}
      \frac{e^{-z}}{x(2\alpha-w)^2} \Big( \frac{wx}{yz}\Big)^\theta
      x^{2-2\theta} dx dw dy dz \\ & \qquad \qquad \qquad \qquad
      \qquad + \frac 1\alpha \int_1^{2\alpha} \int_x^{2\alpha}
      \int_w^{2\alpha} \int_y^{2\alpha}
      \frac{e^{w+x-y-z}}{x(2\alpha-w)^2}dz dy dw dx\Big) \\ & =
      \mathcal O \Big( \frac{1}{\alpha} \int_0^{2\alpha} \int_0^z
      \int_0^y \frac{e^{-z} w^2}{(2\alpha-w)^2} \frac{1}{(yz)^\theta}
      dw dy dz \\ & \qquad \qquad \qquad \qquad \qquad + \frac 1\alpha
      \int_1^{2\alpha} \int_x^{2\alpha} \int_w^{2\alpha}
      \frac{e^{w+x-y}(e^{-y} - e^{-2\alpha})}{x(2\alpha-w)^2} dy dw dx
      \Big)\\ & = \mathcal O \Big( \frac{1}{\alpha^3} \int_0^{2\alpha}
      \int_0^z \int_0^{y\wedge \alpha} \frac{e^{-z}w^2}{(yz)^\theta}dw
      dy dz + e^{-\alpha}\alpha^{1-2\theta} \int_\alpha^{2\alpha}
      \int_w^{2\alpha} \int_w^{2\alpha} \frac{1}{(2\alpha-w)^2} dz dy
      dw \\ & \qquad \qquad \qquad \qquad + \frac 1\alpha
      \int_1^{2\alpha} \int_x^{2\alpha} \frac{e^{w+x}(\tfrac 12
        (e^{-2w} - e^{-4\alpha})
        - e^{-w-2\alpha} + e^{-4\alpha})}{x(2\alpha-w)^2} dw dx\Big) \\
      & \stackrel{w\to 2\alpha-w}{=} \mathcal O\Big(
      \frac{1}{\alpha^3} \int_0^{2\alpha} e^{-z} z^{4-2\theta} dz \\ &
      \qquad \qquad \qquad \qquad \qquad + \frac {e^{-4\alpha}}\alpha
      \int_1^{2\alpha} \int_0^{2\alpha-x} \frac{e^{2\alpha-w+x}
        (\tfrac 12 (e^{2w}-1) -e^w+1)}{xw^2} dw dx \Big) \\ & =
      \mathcal O \Big( \frac{1}{\alpha^3} + \frac{1}{\alpha}
      \int_1^{2\alpha-1} \frac{e^{-2x}}{x(2\alpha-x)^2} dx \Big) =
      \mathcal O\Big( \frac{1}{\alpha^3}\Big).
    \end{aligned}
  \end{equation}
  For $\theta \geq 1$, we compute, similar to \eqref{eq:Var1},
  \begin{align*}
    \int_0^1 \int_0^w & t_{\alpha,\theta}^\ast(w;0)
    t_{\alpha,\theta}^\ast(x;w) \frac{w}{1-w}(1-x) dx dw = \int_0^1
    \int_0^w \int_w^1 \int_w^1
    \frac{e^{{2\alpha}(w+x-y-z)}}{(1-w)^2x} \Big(
    \frac{wx}{yz}\Big)^\theta dz dy dx dw \\ &
    \stackrell{(w,x,y,z)\to}{\leq}{{2\alpha}(w,x,y,z)} \mathcal O
    \Big( \frac 1\alpha \int_0^{2\alpha} \int_0^z \int_0^y \int_0^w
    \frac{e^{w+x-y-z}}{(2\alpha-w)^2 }\frac{w}{yz} dx dw dy dz\Big)
    \\ & = \mathcal O\Big( \frac 1\alpha \int_0^{2\alpha} \int_0^z
    \int_0^y \int_0^{w\wedge 1} \frac{e^{-z}w}{({2\alpha}-w)^2 yz} dx
    dw dy dz \\ & \qquad \qquad \qquad \qquad \qquad \qquad + \frac
    1\alpha \int_1^{2\alpha} \int_x^{2\alpha} \int_w^{2\alpha}
    \int_y^{2\alpha} \frac{e^{w+x-y-z}w}{(2\alpha-w)^2 yz} dz dy dw
    dx\Big) \\ & = \mathcal O\Big( \frac 1{\alpha^3} \int_0^{2\alpha}
    \int_0^z \int_0^{y\wedge \alpha} e^{-z} dw dy dz + \frac 1\alpha
    \int_1^{2\alpha} \int_x^{2\alpha} \int_w^{2\alpha}
    \int_y^{2\alpha} \frac{e^{w+x-y-z}}{(2\alpha-w)^2 x} dz dy dw
    dx\Big) \\ & = \mathcal O\Big( \frac{1}{\alpha^3}\Big),
  \end{align*}
  since the last term in the second to last line equals the term in
  the fifth line of \eqref{eq:sim1a} such that we are done.
\end{proof}

\section{Proof of Theorem \ref{T1}}
\label{sub:T1}
Recall the transition rates of the process $\xi^{\mathcal X}$ given in
Table \ref{tab:2}.  We prove Theorem \ref{T1} in four steps. First, we
establish that almost surely, all lines in $\xi^{\mathcal X}$ are in
the wild-type background by time $\beta_0$. In Step 2, we give an
approximate structured coalescent $\eta^{\mathcal X}$, which has
different rates before and after $\beta_0$. This process already
provides us with a good approximation for $\xi_{\geq\beta_0}$. In Step
3, we will use a random time-change of the diffusion $\mathcal X$ to a
supercritical Feller diffusion $\mathcal Y$ with immigration. In Step
4 we will use facts about the connection of the supercritical
branching process with immigration to a Yule process with immigration.

\subsection{Step 1: All lines in wild-type background by time $\beta_0$}
We will show below that all lines in the structured coalescent
$\xi^{\mathcal X}$ are in the wild-type background by time $\beta_0$.

\begin{proposition}\label{P4}
  For all values of $\theta, \alpha$,
  \begin{align*}
    \mathbb P_{\alpha,\theta}[\xi_{\beta_0}^B = \emptyset] = 1.
  \end{align*}
\end{proposition}
\begin{proof}
  Note that the structured coalescent $\xi^{\mathcal X}$ can be
  constructed using a finite number of processes $\xi_1^{\mathcal X},
  \xi_2^{\mathcal X}, \xi_3^{\mathcal X}, \xi_4^{\mathcal X}$ (compare
  Table \ref{tab:1}). In particular, the escape of lines in the
  beneficial background to the wild-type background due to mutation is
  given by the processes $\xi_2^{\mathcal X}$. Moreover, we know from
  Lemma \ref{l:est1a} that any line in the beneficial background by
  time $T_0+\varepsilon$ for some $\varepsilon>0$ will experience such
  an escape since $\xi_2^{\mathcal X}\cap [T_0;T_0+]\neq \emptyset$
  almost surely. Hence the assertion follows.
\end{proof}

\subsection{Step 2: Approximation of $\xi_{\geq 0}$ by $\eta_{\geq 0}$}
In order to define the process $\eta^{\mathcal X}$ we use transition
rates as given in Tables \ref{tab:3} and \ref{tab:4}. Moreover, set
$$\eta_{\geq s} := \int \mathbb P_{\alpha,\theta}[d\mathcal X]
(\eta_{s+t}^{\mathcal X})_{t\geq 0}.$$ We will establish that
$\xi_{\geq 0}$ and $\eta_{\geq 0}$ are close in variational distance.

\begin{proposition}\label{P5}
  The bound
  \begin{align*}
    d_{TV}(\xi_{\geq 0}, \eta_{\geq 0}) = \mathcal O \Big(
    \frac{1}{(\log\alpha)^2}\Big).
  \end{align*}
  holds in the limit of large $\alpha$ and uniformly on compacta in
  $n, \gamma$ and $\theta$.
\end{proposition}

\begin{remark}
  Note that $(\eta^{\mathcal X}_t)_{t \geq \beta_0}$ does not depend
  on $\mathcal X$ (i.e. $\eta_{\geq \beta_0}=(\eta^{\mathcal X}_t)_{t
    \geq \beta_0}$ in distribution for all realizations of $\mathcal
  X$). Using the same argument as in Step 1 all lines of
  $\eta_{\beta_0}$ are in the wild-type background. These two facts
  together imply that $\xi_{\geq \beta_0}$ approximately has the same
  transition rates as the finite Kingman coalescent $\mathcal C$,
  which is the statement of \eqref{eq:bound2}.
\end{remark}
 
\begin{table}
  \begin{center}
    \begin{tabular}{|c|ccccc|}\hline
      \rule[-3mm]{0cm}{1cm}event & \hspace{0.5ex} coal in $B$ \hspace{0.5ex}& \hspace{0.5ex} coal in $b$ \hspace{0.5ex} &  \hspace{0.5ex} mut from $B$ to $b$\hspace{0.5ex} & \hspace{0.5ex} rec from $B$ to $b$ \hspace{0.5ex} & \hspace{0.5ex} rec from $b$ to $B$   \\[2ex]\hline
      \rule[-3mm]{0cm}{1cm}rate & $\frac{1-X_t}{X_t}$ & $0$ & $\frac \theta 2\frac{1-X_t}{X_t}$ & $\rho(1-X_t)$ &  $0$ \\\hline
\end{tabular}
\end{center}
\caption{\label{tab:3}Transition rates of $\eta^{\mathcal X}$ in the interval
  $[0;\beta_0]$.}
\end{table}

\begin{table}
  \begin{center}
    \begin{tabular}{|c|ccccc|}\hline
      \rule[-3mm]{0cm}{1cm}event & \hspace{0.5ex} coal in $B$ \hspace{0.5ex}& \hspace{0.5ex} coal in $b$ \hspace{0.5ex} &  \hspace{0.5ex} mut from $B$ to $b$\hspace{0.5ex} & \hspace{0.5ex} rec from $B$ to $b$ \hspace{0.5ex} & \hspace{0.5ex} rec from $b$ to $B$   \\[2ex]\hline
      \rule[-3mm]{0cm}{1cm}rate & $0$ & $1$ & $0$ & $0$ &  $0$ \\\hline
\end{tabular}
\end{center}
\caption{\label{tab:4}Transition rates of $\eta^{\mathcal X}$ in the interval
  $[\beta_0;\infty]$.}
\end{table}
\begin{proof}[Proof of Proposition \ref{P5}]
  Again it is important to note that $\xi^{\mathcal X}$ can be
  constructed using a finite number of Poisson processes
  $\xi_1^{\mathcal X}, ..., \xi_6^{\mathcal X}$. In the same way,
  $\eta^{\mathcal X}$ can be constructed using a finite number of
  Poisson processes $\xi_1^{\mathcal X}, \xi_2^{\mathcal X},
  \xi_3^{\mathcal X}, \xi_5^{\mathcal X}$ and Poisson processes with
  rates $\frac{1-X_t}{X_t}$.

  Consider times $0\leq \beta\leq \beta_0$ first {\color{black} and
    recall $T_i^{\mathcal X} = \sup\xi_i^{\mathcal X}$.} A single line
  may escape the beneficial background and recombine back in
  $\xi^{\mathcal X}$, while this is not possible in $\eta^{\mathcal
    X}$. Such an event in $\xi^{\mathcal X}$ requires that either
  $\xi_4^{\mathcal X}\cap [T_0; T_2^{\mathcal X}] \neq \emptyset$ or
  $\xi_4^{\mathcal X}\cap [T_0; T_3^{\mathcal X}] \neq \emptyset$ for
  one triple of the processes $\xi_2^{\mathcal X}, \xi_3^{\mathcal X},
  \xi_4^{\mathcal X}$, which has a probability of order $\mathcal
  O\Big( \frac{1}{(\log\alpha)^2}\Big)$ by \eqref{eq:l:est2:1} and
  \eqref{eq:l:est2:1b}. Hence, ignoring these events produces a total
  variation distance of at most $\mathcal O\Big(
  \frac{1}{(\log\alpha)^2}\Big)$. The coalescence rates in the
  beneficial background of the processes $\xi^{\mathcal X}$ and
  $\eta^{\mathcal X}$ differ by 1. By the bound
  \eqref{eq:l:est2:2}, the different coalescence rates in the
  beneficial background produce a total variation distance of
  $\mathcal O\big( \frac{\log\alpha}{\alpha}\big)$. Lastly, since
  $\frac 1{1-X_t} = 1 + \frac{X_t}{1-X_t}$, we can assume that
  coalescences in the wild-type background in $\xi^{\mathcal X}$ occur
  along events of one pair of processes $\xi_5^{\mathcal
    X}\cup\xi_6^{\mathcal X}$. Such an event requires that either
  $\xi_5^{\mathcal X}\cap [T_0;T] \neq\emptyset$, $\xi_6^{\mathcal X}
  \cap [T_0;T_2^{\mathcal X}]\neq \emptyset$ or $\xi_6^{\mathcal X}
  \cap [T_0;T_3^{\mathcal X}]\neq \emptyset$. These events together
  have a probability of order $\mathcal O\big(
  \frac{1}{\sqrt\alpha}\big)$ by \eqref{eq:l:est2:2},
  \eqref{eq:l:est2:3} and \eqref{eq:l:est2:3b} and hence, ignoring
  these events gives a total variation distance of order $\mathcal
  O\big( \frac{1}{\sqrt\alpha}\big)$. Hence, $\xi^{\mathcal X}$ and
  $\eta^{\mathcal X}$ are close for times $0\leq \beta\leq \beta_0$.

  Let us turn to times $\beta\geq \beta_0$. It is important to notice
  that, using the same arguments as in the proof of Proposition
  \ref{P4}, $\mathbb P_{\alpha,\theta}[\eta_{\beta_0}^B=\emptyset]=1$.
  Note that $\eta^{\mathcal X}$ differs from $\xi^{\mathcal X}$ by
  ignoring back-recombinations along processes $\xi^{\mathcal X}_4$
  and by changing the coalescence rate in the wild-type background from
  $\frac{1}{1-X_t}$ to 1. Considering a single line, ignoring events
  in $\xi^{\mathcal X}_4$ produces a total variation distance of order
  $\mathcal O\big( \frac{1}{\alpha\log\alpha}\big)$ by
  \eqref{eq:l:est1:1}. Hence, we can assume that all lineages are in
  the wild-type background for $\beta\geq \beta_0$. For coalescences
  in the wild-type background, we are using that $\frac 1{1-X_t} = 1 +
  \frac{X_t}{1-X_t}$ and the fact that ignoring events, which occur
  along one process $\xi_6^{\mathcal X}$ produces a total variation
  distance of order $\mathcal O\big( \frac{1}{\alpha^2}\big)$ by
  \eqref{eq:l:est1:2}.

  Putting all arguments together, we have
  $$ d_{TV}(\xi_{\geq 0}, \eta_{\geq 0}) \leq 
  d_{TV}(\xi_{0\leq \beta\leq \beta_0}, \eta_{0\leq \beta\leq
    \beta_0}) + d_{TV}(\xi_{\geq \beta_0}, \eta_{\geq \beta_0}) =
  \mathcal O\Big( \frac{1}{(\log\alpha)^2}\Big).$$
\end{proof}

\subsection{Step 3: Random time-change to a supercritical branching process}
By a random time change, the diffusion \eqref{eq:diff} is taken to a
supercritical branching process with immigration. Specifically, use
the random time change $d\tau = (1-X_t)dt$ to see that the
time-changed process $\mathcal Y = (Y_\tau)_{\tau\geq 0}$ solves
\begin{align}\label{eq:diff3}
  dY = \big(\tfrac \theta 2 + \alpha Y\big) d\tau + \sqrt{Y}
  d\widetilde W,
\end{align}
stopped when $Y_\tau=1$, with some Brownian motion $(\widetilde
W_\tau)_{\tau\geq 0}$ (see e.g. \cite{EthierKurtz1986}, Theorem
6.1.3). Hence, $\mathcal Y$ is a supercritical branching process with
immigration.  Analogous to $T_0$ and $T$, define the random times
\begin{align*}
  \widetilde T_0 := \sup\{\tau \geq 0: Y_\tau =0\},\qquad \qquad
  \widetilde T := \inf\{\tau \geq 0: Y_\tau =1\}
\end{align*}
as well as
$$ \widetilde \beta:= \widetilde T - \tau, \qquad \qquad \widetilde \beta_0 := \widetilde T - \widetilde T_0.$$
Conditioned on $\mathcal Y$, we define the structured coalescent
$\zeta^{\mathcal Y}:=(\zeta^{\mathcal Y}_{\widetilde \beta})_{0\leq
  \widetilde \beta\leq \widetilde \beta_0}$ with transition rates
defined in Table \ref{tab:5}. Setting
$$ \zeta_{\widetilde\beta_0} := \int \mathbb P[d\mathcal Y] \zeta_{\widetilde\beta_0}^{\mathcal Y}$$
we immediately obtain the following result.

\begin{table}
  \begin{center}
    \begin{tabular}{|c|ccccc|}\hline
      \rule[-3mm]{0cm}{1cm}event & \hspace{0.5ex} coal in $B$ \hspace{0.5ex}& \hspace{0.5ex} coal in $b$ \hspace{0.5ex} &  \hspace{0.5ex} mut from $B$ to $b$\hspace{0.5ex} & \hspace{0.5ex} rec from $B$ to $b$ \hspace{0.5ex} & \hspace{0.5ex} rec from $b$ to $B$   \\[2ex]\hline
      \rule[-3mm]{0cm}{1cm}rate & $\frac{1}{X_t}$ & $0$ & $\frac \theta 2\frac{1}{X_t}$ & $\rho$ &  $0$ \\\hline
    \end{tabular}
  \end{center}
  \caption{\label{tab:5}Transition rates of $\zeta^{\mathcal Y}$.}
\end{table}

\begin{proposition}\label{P6}
  For all $\theta, \alpha$ and $\gamma$,
  \begin{align*}
    d_{TV}(\zeta_{\widetilde\beta_0}, \eta_{\beta_0}) = 0.
  \end{align*}
\end{proposition}
\begin{proof}
  The pairs $(\mathcal X, \eta^{\mathcal X})$ and $(\mathcal Y,
  \zeta^{\mathcal Y})$ can be perfectly coupled by setting $d\tau =
  (1-X_t)dt$. Under this random time change $\beta_0$ becomes
  $\widetilde \beta_0$ and hence, the averaged processes
  $\eta_{\beta_0}$ and $\zeta_{\widetilde\beta_0}$ can also be
  perfectly coupled, leading to a distance of 0 in total variation.
\end{proof}

\subsection{Step 4: Genealogy of $\mathcal Y$ is $\Upsilon$}

\begin{proposition}
  \label{P7}
  Let $\mathcal Y$ be a supercritical Feller branching process
  governed by \eqref{eq:diff3} started in $0$ and let
  $\widetilde{\mathcal F}^{\mathcal Y}$ be the forest of individuals
  with infinite descent. Then the following statements are true:
  \begin{enumerate}
  \item $\widetilde{\mathcal F} = \int \mathbb P[d\mathcal Y]
    \widetilde{\mathcal F}^{\mathcal Y}$ is a Yule tree with birth
    rate $\alpha$ and immigration rate $\alpha\theta$.
  \item The number of lines in $\widetilde{\mathcal F}$ extant at time
    $\widetilde T$ (when $\mathcal Y$ hits 1 for the first time) has a
    Poisson distribution with mean $2\alpha$.
  \item Given $\mathcal Y$, the pair coalescence rate of
    $\widetilde{\mathcal F}^{\mathcal Y}$ is $1/Y_\tau$ and the rate
    by which migrants occur is $\frac \theta 2 \frac{1}{Y_\tau}$.
  \end{enumerate}
\end{proposition}
\begin{proof}
  The proposition is analogous to Lemma 4.5 of
  \cite{EtheridgePfaffelhuberWakolbinger2006} and can be proved along
  similar lines. We give an alternative proof based on an
  approximation of $\mathcal Y$ by finite models.

  Statement 1.\ is an extension of Theorem 3.2 of
  \cite{OConnell1993}. Consider a time-continuous supercritical
  Galton-Watson process $\mathcal Y^N=(Y^N_t)_{t\geq 0}$ with
  immigration, starting with 0 individuals. Each individual branches
  after an exponential waiting time with rate $N$. (Note that $N$ is a
  scaling parameter and not directly related to the population size.)
  It splits in two or dies with probabilities $\frac{1+s}{2}$ and
  $\frac{1-s}{2}$, respectively. New lines enter the population at
  rate $\frac{\theta N}{2}$. Then, $\mathcal Y^N/N \Rightarrow
  \mathcal Y$, the solution of \eqref{eq:diff3} as $N\to\infty$, if
  $Ns\xrightarrow{N\to\infty}\alpha$.  Moreover, the probability that
  an individual of the population has an infinite line of descent is
  $2s+\mathcal O(s^2)$ for small $s$. As a consequence, the rate of
  immigration of individuals with an infinite line of descent is
  $\theta \alpha$ in $\mathcal Y$. In addition, each such line has
  descendants, which have an infinite line of descent. In particular,
  each immigrant with an infinite line of descent is founder of a Yule
  tree with branching rate $\alpha$; see \cite{OConnell1993}.

  For 2., consider times $t$ when $Y_t^N/N=1$, i.e., $Y^N_t=N$ for the
  first time. Since all lines have an infinite number of offspring
  independently of each other, each with probability $2s+\mathcal
  O(s^2)$, the total number of lines with infinite descent is
  binomially distributed with parameters $N$ and $2s+\mathcal
  O(s^2))$. In the limit $N\to \infty$, this becomes a Poisson number
  of lines in $\widetilde{\mathcal F}$ with parameter $2\alpha$ at
  times $t$ when $Y_t=1$.

  For 3., let $Y^N_\tau=y^N$ such that
  $y^N/N\xrightarrow{N\to\infty}y$. Note that by exchangeability the
  coalescence and mutation rates are the same for lines of finite and
  infinite descent. Since $\mathcal Y^N/N$ converges to a diffusion
  process, we can assume that $\sup_{\tau-1/N\leq s\leq \tau}
  |Y^N_{s}-y^N|=\mathcal O(\sqrt{N})$. Consider the emergence of a
  migrant first and recall that migrants enter the population at rate
  $\frac{\theta N}{2}$, independent of $Y^N_\tau$. Since we pick a
  specific line among all $y^N$ lines with probability $1/y^N$, that
  rate of immigration for times $[\tau-1/N;\tau]$ is $\frac{\theta
    N}{2y^N+\mathcal O(\sqrt{N})} \xrightarrow{N\to\infty}
  \frac{\theta}{2}\frac{1}{y}$. Next, turn to coalescence of a pair of
  lines. Observe that such events may only occur along birth events
  forward in time, which occur at rate $Ny^N\frac{1+s}{2}$. Since the
  probability that a specific pair out of $y^N$ lines coalesces is
  $1/\binom{y^N}{2}$ we find that the coalescence rate for times
  $[\tau-1/N;\tau]$ is
  $$ N(y^N+\mathcal O(\sqrt{N})) \frac{1+s}{2} \frac{1}{\binom{y^N+\mathcal O(\sqrt{N})}{2}} 
  \xrightarrow{N\to\infty} \frac{1}{y}.$$ Hence we are done.

\end{proof}

\begin{proposition}
  \label{P8}
  The bound
  \begin{align*}
    d_{TV} ( \zeta_{\widetilde\beta_0}, \Upsilon) = \mathcal O \Big(
    \frac{1}{(\log\alpha)^2}\Big)
  \end{align*}
  holds for large $\alpha$ and is uniform on compacta in $n, \gamma$
  and $\theta$.
\end{proposition}
\begin{proof}
  The statement as well as its proof is analogous to Proposition 4.7
  in \cite{EtheridgePfaffelhuberWakolbinger2006}. By Proposition
  \ref{P7}, the random partition $\zeta_{\widetilde\beta_0}$ arises by
  picking $n$ lines from the tips of a Yule tree with birth rate
  $\alpha$ with immigration rate $\alpha\theta$ and which has grown to
  a Poisson$(2\alpha)$ number of lines, and marking all lines at
  constant rate $\rho$. Hence, the difference of
  $\zeta_{\widetilde\beta_0}$ and $\Upsilon$ arises from

  ~~

  \parbox[t]{6cm}{\hspace{1cm}$\zeta_{\widetilde\beta_0}$:
  \begin{enumerate}
  \item picking from a Yule tree with Poisson$(2\alpha)$ tips
  \item a constant marking rate $\rho$ for all lines
  \end{enumerate}}
\parbox[t]{6cm}{\hspace{1cm} $\Upsilon$: \begin{enumerate} \item[1'.] picking
    from a Yule tree with $\lfloor 2\alpha\rfloor$ tips
    \item[2'.] a marking probability of $1-p_{i_1}^{i_2}(\gamma,\theta)$ for a 
      branch, which starts at Yule-time $i_1$ and ends at Yule-time $i_2$. 
    \end{enumerate}}

  ~~

  \noindent Both differences only have an effect if they lead to
  different marks of the Yule tree with immigration. To bound the
  probability of the difference of 1. and 1'., note that the Poisson
  distribution has a variance of $2\alpha$ and hence, typical
  deviations are of the order $\sqrt\alpha$. Given such a typical
  deviation of the Poisson from its mean, the probability of a
  different marking of both Yule trees is of the order $\mathcal
  O\big( \frac{1}{\sqrt{\alpha}\log\alpha}\big)$, as shown below (4.9)
  in \cite{EtheridgePfaffelhuberWakolbinger2006}. For the different
  marks from 2. and 2'. note first that the probability that two marks
  occur within any Yule-time is, since the marks and splits of the
  Yule tree having competing exponential distributions, bounded by
  $$ \sum_{i=1}^{\lfloor 2\alpha\rfloor} \Big(\frac{\rho}{\alpha (i + \theta) + 
    \rho}\Big)^2 \leq \frac{\gamma^2}{(\log\alpha)^2}
  \sum_{i=1}^\infty \frac{1}{i^2} = \mathcal O\Big(
  \frac{1}{(\log\alpha)^2}\Big).$$ Hence, treating these double
  hits of Yule times differently only leads to a total variation
  distance of $\mathcal O\big(\frac{1}{(\log\alpha)^2}\big)$. In
  particular, we may mark all lines of the Yule tree independently (as
  in $\Upsilon$) since dependence of marks only arises by double hits
  of Yule times. The probability that a line that starts in Yule time
  $i_1$ and ends in Yule-time $i_2$ is not marked, is, again using
  competing exponentials,
  \begin{align*} \prod_{j=i_1+1}^{i_2} \frac{\alpha (j +
      \theta)}{\alpha (j + \theta) + \rho} & = \prod_{j=i_1+1}^{i_2}
    \Big( \exp\Big( - \frac{\gamma/\log\alpha}{j+\theta +
      \gamma/\log\alpha}\Big) + \frac{1}{j^2} \mathcal O\Big(
    \frac{1}{(\log\alpha)^2}\Big)\Big) \\ & = \exp\Big( -
    \frac{\gamma}{\log\alpha}\sum_{j=i_1+1}^{i_2}\frac 1{j+\theta +
      \gamma/\log\alpha}\Big) + \mathcal
    O\Big(\frac{1}{(\log\alpha)^2}\Big) \\ & =
    p_{i_1}^{i_2}(\gamma,\theta) + \mathcal
    O\Big(\frac{1}{(\log\alpha)^2}\Big).
  \end{align*}
  Hence, the difference of 2. and 2.' accounts for a total variation
  distance of oder $\mathcal O\big(\frac{1}{(\log\alpha)^2}\big)$ and
  we are done.

\end{proof}

\subsection{Conclusion}
Using Propositions \ref{P4}-\ref{P8} we can now prove Theorem
\ref{T1}. Note that \eqref{eq:alsu} is the same statement as given in
Proposition \ref{P4}. Since $\xi_{\beta_0}^B=\emptyset$ almost surely,
all ancestral lines of $\xi_{\beta_0}$ must be in the wild-type
background and so, using Propositions \ref{P5}, \ref{P6} and \ref{P8},
$$ d_{TV}(\xi_{\beta_0}^b, \Upsilon) \leq d_{TV}(\xi_{\geq 0}, \eta_{\geq 0}) 
+ d_{TV}(\eta_{\beta_0}, \zeta_{\widetilde \beta_0}) +
d_{TV}(\zeta_{\widetilde \beta_0},\Upsilon) = \mathcal O \Big(
\frac{1}{(\log\alpha)^2}\Big).$$ For the approximation of $\xi^b_{\geq
  \beta_0}$ by the finite Kingman coalescent $\mathcal C$ we will use
Proposition \ref{P5}. First, note that by the same reasoning as in the
proof of Proposition \ref{P4}, $\mathbb P[\eta_{\beta_0}^B \neq
\emptyset] = 0$. Moreover, $\xi^B_{\geq \beta_0}\neq
(\emptyset)_{t\geq 0}$ requires a back-recombination event with rate
$\rho X_t$ for some time $0\leq t \leq T_0$ and thus, using
\eqref{eq:l:est1:1},
$$ \mathbb P[\xi^B_{\geq \beta_0}\neq (\emptyset)_{t\geq 0}] \leq \mathcal O\Big( \frac{1}{\alpha\log\alpha}\Big).$$
Let $\mathcal C'$ be a finite Kingman coalescent that starts with a
random number of lines and which is distributed like
$\xi_{\beta_0}^b$. Then, since $d_{TV}(\xi_{\geq \beta_0}^b, \mathcal
C') \leq d_{TV}(\xi_{\geq 0}, \eta_{\geq 0})$,
$$ d_{TV}(\xi_{\geq \beta_0}^b, \Upsilon\circ \mathcal C) \leq 
d_{TV}(\xi_{\beta_0}^b, \Upsilon) + d_{TV}(\xi_{\geq \beta_0}^b,
\mathcal C') = \mathcal O\Big( \frac{1}{(\log\alpha)^2}\Big).
$$

\subsection{Sampling at time $t<T$}
\label{rem7}
Assume $t<T$ is such that $X_t = 1 - \delta/\log\alpha$ for some
$\delta>0$. To approximate the number of recombination events in
$[t;T]$, we can use the time-rescaling to the process $\mathcal Y$
from \eqref{eq:diff3} and Proposition \ref{P7} to note that the Yule
process has a Poisson number with parameter $2\alpha( 1 -
\delta/\log\alpha)$ lines at the time the supercritical branching
process has $Y_\tau = 1 - \delta/\log\alpha$. Since recombination
events fall on the Yule tree at constant rate $\rho$, the probability
of such an event during $[\tau;\widetilde T]$ is
$$ \frac{\rho}{\alpha} \sum_{i=\lfloor 2\alpha(1-\delta/\log\alpha)\rfloor}^{\lfloor 2\alpha\rfloor} \frac 1i 
= \mathcal O \Big( \frac{1}{\log\alpha} \log \Big( \log\Big( 1 -
\delta/\log\alpha\Big)\Big) = \mathcal O\Big(
\frac{1}{(\log\alpha)^2}\Big).$$ A similar calculation shows that
there are no coalescence events in a sample from the Yule tree between
Yule times $\lfloor 2\alpha(1-\delta/\log\alpha)\rfloor$ and $\lfloor
2\alpha\rfloor$ with high probability.

~~

\noindent{\bf Acknowledgement} We thank John Wakeley for fruitful
discussion and an anonymous referee for a careful reading of our
manuscript. The Erwin-Schr\"odinger Institut, Vienna is acknowledged
for its hospitality during the Workshop \emph{Frontiers of
  Mathematical Biology} in April, 2008. Both authors were supported by
the Vienna Science and Technology Fund WWTF. PP obtained additional
support by the BMBF, Germany, through FRISYS (Freiburg Initiative for
Systems biology), Kennzeichen 0313921.


\end{document}